\documentclass[10pt,a4paper]{article}
\usepackage{hyperref}
\hypersetup{
	colorlinks=true,
	linkcolor=blue,
	filecolor=magenta,      
	urlcolor=cyan,
}
\usepackage{a4wide}
\usepackage{graphicx,xcolor,textpos}
\usepackage{helvet}
\usepackage[utf8]{inputenc}
\usepackage[T1]{fontenc}
\usepackage[english]{babel}
\usepackage{amsmath}
\usepackage{amsfonts}
\usepackage{amssymb,amsthm}
\usepackage{graphicx}
\usepackage{mathtools}
\usepackage{bbold}
\usepackage{csquotes}
\usepackage{tikz-cd}
\usepackage{cases}
\usepackage{listings}
\usepackage{algorithm}
\usepackage{algorithmic}
\usepackage{enumerate}

\newtheorem{thm}{Theorem}[section]
\newtheorem{thmx}{Theorem}

\newtheorem{lemma}[thm]{Lemma}
\newtheorem{prop}[thm]{Proposition}

\newtheorem{cor}[thm]{Corollary}

\theoremstyle{definition}
\newtheorem{df}[thm]{Definition}
\newtheorem{nota}[thm]{Notation}

\theoremstyle{remark}
\newtheorem*{rem}{Remark}
\theoremstyle{definition}
\newtheorem{ex}[thm]{Example}
\theoremstyle{definition}
\newcommand{\Z}{\mathbb{Z}}
\newcommand{\N}{\mathbb{N}}
\newcommand{\Q}{\mathbb{Q}}
\newcommand{\R}{\mathbb{R}}
\newcommand{\C}{\mathbb{C}}
\newcommand{\F}{\mathbb{F}}
\newcommand{\A}{{\overline{\mathbb{Q}}}}
\newcommand{\norm}[1]{\|#1\|}

\DeclareMathOperator{\Aut}{Aut}
\DeclareMathOperator{\Inn}{Inn}
\DeclareMathOperator{\Out}{Out}

\DeclareMathOperator{\Gal}{Gal}
\DeclareMathOperator{\lcm}{lcm}

\DeclareMathOperator{\vct}{span}
\DeclareMathOperator{\GL}{GL}
\DeclareMathOperator{\RF}{RF}
\DeclareMathOperator{\Id}{Id}

\DeclareMathOperator{\core}{Core}

\DeclareMathOperator{\Gui}{Gui}
\DeclareMathOperator{\nom}{norm}
\DeclareMathOperator{\rng}{ring}

\title{Residual Finiteness Growth in Virtually Nilpotent Groups}
\author{Jonas Der\'e, Joren Matthys and Lukas Vandeputte\thanks{KU Leuven Campus Kulak Kortrijk, Department of Mathematics, Research unit `Algebraic Topology and Group Theory', B-8560 Kortrijk, Belgium. The authors were supported by Internal Funds KU Leuven (project number 3E220559). The third author kindly acknowledges the support by the group of Science Engineering and Technology at KU Leuven Campus Kulak.}}
\date{}
\begin{document}
	\maketitle
	\begin{abstract}
	The residual finiteness growth $\RF_G: \N \to \N$ of a finitely generated group $G$ is a function that gives the smallest value of the index $[G:N]$ with $N$ a normal subgroup not containing a non-trivial element $g$, in function of the word norm of that element $g$. It has been studied for several classes of finitely generated groups, including free groups, linear groups and virtually abelian groups. This paper shows that if $G$ is virtually nilpotent, then $\RF_G = \log^\delta$ for some $\delta\in \N\cup\{0\}$, with moreover an explicit formula for $\delta$ in terms of Lie algebras. This implies in particular that it is an invariant of the complex Mal'cev completion, leading to the application that residual finiteness growth is a profinite invariant for virtually nilpotent groups.
	\end{abstract}
	\section{Introduction}
	A group $G$ is said to be residually finite if the intersection of all finite index normal subgroups is trivial. Hence, for every non-trivial element $g\in G$, there exists a finite index normal subgroup $N_g\lhd_f G$ such that $g\notin N_g$. Equivalently, there exists a homomorphism $\varphi_g: G \to Q_g$ to a finite group such that $\varphi_g(g) \neq e$ by taking $Q_g = G/N_g$. The residual finiteness growth $\RF_G: \N \to \N$ is a function that quantifies this property for finitely generated residually finite groups. If $\norm{\cdot}_G$ denotes a word norm on $G$, then $\RF_G$ is the smallest function such that for every $g\in G$ with $0\neq \norm{g}_G \leq n$ there exists $N_g \lhd_f G$ with $g\notin N_g$ and $[G:N_g] \leq \RF_G(n)$.
	
	The study of the residual finiteness growth started in the founding paper \cite{bou2010quantifying}. Since then, this function has been studied for various classes, including free groups \cite{bradford2019short}, linear groups \cite{franz2017quantifying}, virtually solvable minimax groups \cite{dere2025residual}, and certain branch groups \cite{branch_groups}. More information concerning the current state of this field can be found in the survey article \cite{survey2022}. This paper focuses on the class of finitely generated virtually nilpotent groups. 
	
	The residual finiteness growth of virtually nilpotent groups was first studied in the papers \cite{MR2861528,bou2010quantifying}, where the authors showed that $\RF_G \preceq \log^{h(G)}$ with $h(G)$ the Hirsch length of a nilpotent group $G$. For virtually nilpotent groups $\Gamma$, the bound $\RF_{\Gamma} \preceq (\RF_G)^{[\Gamma:G]}$ with $G$ a finite index nilpotent subgroup was used to show that they have a polylogarithmic upper bound. In fact, these are the only groups with this property within the class of finitely generated linear groups \cite{MR2861528}. In \cite[Conjecture 1]{dere2025twostep}, the authors conjectured that $\RF_G$ is polylogarithmic if and only if $G$ is virtually nilpotent.
	
	Later, the upper bounds on the residual finiteness growth in the nilpotent case were improved in \cite{MarkPubl}. Actually, the paper claimed exact results for all nilpotent groups, but the proof was erroneous, and the corrigendum \cite{pengitoreJGT} could only re-establish the upper bound as counterexamples showed that the lower bound was false. The derived bound was an invariant of the rational Mal'cev completion and hence commensurable nilpotent groups had the same upper bound. In general, it was expected that commensurable nilpotent groups actually have the same residual finiteness growth, but the author only established the equality $\RF_G = \RF_{G/T}$ with  $T$ the torsion subgroup.
	
	In \cite{math_virt_ab}, the residual finiteness growth of virtually abelian groups was calculated. This paper had two main takeaways. Firstly, it noted that if there was a short exact sequence of the form
	$$1 \to G \to \Gamma \to H \to 1$$
	with $H \cong \Gamma/G$ finite, then $\RF_{\Gamma}$ was equal to $\RF_{G, \mathcal{P}_{H\curvearrowright G}}$, i.e.~the residual finiteness growth of $G$ when restricting to normal subgroups that are invariant under the induced action $\varphi: H \to \Out(G)$. Exploiting this in the virtually abelian setting revealed that $\RF_{\Gamma}$ could be calculated using invariant subspaces of a maximal free abelian subgroup $\Z^n$. Secondly, the paper showed that $\RF_{\Gamma}$ was given by $\log^\delta$ with $\delta$ the largest dimension of an invariant $\C$-subspace of $\C^n$, indicating that the residual finiteness growth depends only on complex completions rather than rational ones. The paper \cite{dere2025twostep} then gave upper bounds for two-step nilpotent groups, successfully improving the previous upper bound and establishing lower bound statement for groups with commutator subgroups of rank at most $2$. Subsequently, the bound for the virtually abelian and two-step nilpotent groups were combined to one uniting upper bound statement of the form $\RF_\Gamma \preceq \log^\delta$ for all virtually nilpotent groups in \cite{dere2025residual}, using the Lie group and Lie algebra correspondence.

This paper completely determines the residual finiteness growth for virtually nilpotent groups, and therefore establishes one implication of \cite[Conjecture 1]{dere2025twostep}.

\begin{thmx} \label{thm_intro_exact}
	Let $\Gamma$ be a finitely generated virtually nilpotent group, then there exists $\delta \in \N \cup \{0\}$ such that $\RF_\Gamma = \log^\delta$.
\end{thmx}
\noindent The upper bound statement will not be proven in full detail, since this is part of \cite{dere2025residual}.
The value $\delta$ can be determined using the corresponding complex Lie algebra and the induced action of a finite quotient, as we will illustrate. In particular, this will imply the following consequence.

\begin{thmx} \label{thm_intro_profinite_invariant}
Let $\Gamma_1,\Gamma_2$ be finitely generated virtually nilpotent groups with isomorphic profinite completions $\hat{\Gamma}_1 \approx \hat{\Gamma}_2$, then $\RF_{\Gamma_1} = \RF_{\Gamma_2}$.
\end{thmx}

In order to prove the main result, we will show in Section \ref{sec_invariance} that the asymptotic behavior of $\RF_{G, \mathcal{P}_{H\curvearrowright G}}$ is preserved when passing from $G$ to an $H$-invariant subgroup.
\begin{thmx} \label{thm_intro_Q_inv2}
	Let $G_1 \lhd_f G_2$ be two torsion-free, finitely generated nilpotent groups and $H \leq \Aut(G_2)$. If $G_1$ is $H$-invariant, then $\RF_{G_1, \mathcal{P}_{H\curvearrowright G_1}} = \RF_{G_2, \mathcal{P}_{H\curvearrowright G_2}}$.
\end{thmx}
\noindent Specifically for nilpotent groups, this result directly implies that commensurable nilpotent groups have the same residual finiteness growth. This commensurability invariance does not hold in the larger class of virtually nilpotent groups, as is clear from \cite{math_virt_ab}.
	
	In Section \ref{sec_Lie_inv}, we will show that $G_1$ in Theorem \ref{thm_intro_Q_inv2} can be chosen to be an LR-group, i.e. under the Lie group-Lie algebra correspondance $G_1$ is also a Lie ring $L$. For such LR-groups, we relate the residual finiteness growth to a similar notion for Lie rings, which we introduce in this paper.  
	\begin{thmx}\label{thm_intro_RFG_RFL}
		Let $G$ be an LR-group with underlying Lie ring structure $L$ and $H \leq \Aut(G)$, then $\RF_{G, \mathcal{P}_{H\curvearrowright G}} = \RF_{L, \mathcal{P}_{H\curvearrowright L}}$.
	\end{thmx}
\noindent Here, $\RF_{L, \mathcal{P}_{H\curvearrowright L}}$ is defined using a Guivarc'h length on the Lie ring $L$ and by using invariant ideals in $L$ instead of normal subgroups. From this result, we conclude that $\RF_{\Gamma}$ for any virtually nilpotent group $\Gamma$ equals $\RF_{L, \mathcal{P}_{H\curvearrowright L}}$ for some Lie ring $L$ and finitely generated $H \leq \Aut(L)$.
	
Finally, Section \ref{sec_upper} gives the proof of the main result using the relation to Lie rings. After recalling the upper bound, we demonstrate that $\log^\delta$ gives a lower bound for $\RF_{L, \mathcal{P}_{H\curvearrowright L}}$ using the Lefschetz Principle and the Ax-Kochen Principle from first order logic. The final section gives the proof of Theorem \ref{thm_intro_profinite_invariant} using the explicit form of $\delta$.
	
	\section{Preliminaries} \label{sec_prelim}
This section introduces notation about two different topics. Firstly, we recall some notions concerning nilpotent groups and Lie algebras, secondly, we introduce the residual finiteness growth for finitely generated residually finite groups and for Lie rings.
	\subsection{Lie Groups and Algebras}
	\label{sec_Lie}
	In this subsection, we introduce the notions of nilpotency and the correspondence between Lie groups and algebras. We will use the convention that $\N = \{1,2,3, \ldots\}$.
	\begin{df}
		Let $R$ be a commutative ring. A \emph{Lie $R$-algebra} $L$ is an $R$-module equipped with an alternating bilinear map $[\cdot, \cdot]_L: L\times L \to L$ that satisfies the Jacobi identity, i.e.
		$$[v_1,[v_2,v_3]_L]_L  + [v_2,[v_3,v_1]_L]_L + [v_3,[v_1,v_2]_L]_L = 0.$$ 
		If $L$ is a free $\Z$-module of finite rank, then we call $L$ a Lie ring.
	\end{df}
If $L$ is a Lie $R$-algebra, then its lower central series $(\gamma_i(L))_{i\in \N}$ is defined as $\gamma_1(L) = L$ and $\gamma_{i+1}(L) = [\gamma_i(L),L]_L$. If there exists $c\in\N$ such that $\gamma_{c+1}(L) = 0$, then we say $L$ is nilpotent. The smallest such $c\in\N$ is called the \emph{nilpotency class} of $L$. Similarly, if $G$ is a group, we define its lower central series $(\Gamma_i(G))_{i\in \N}$ via $\Gamma_1(G) = G$ and $\Gamma_{i+1}(G) = [\Gamma_i(G),G] = \langle [g,h] \mid g\in \Gamma_i(G), h\in G \rangle$, where $[g,h] := g^{-1}h^{-1}gh$. If there exists $c\in\N$ such that $\Gamma_{c+1}(G) = \{e\}$, then we say $L$ is \emph{nilpotent}. The smallest such $c\in\N$ is called the \emph{nilpotency class} of $G$.

	\begin{df}
		A \emph{$\mathcal{I}$-group} denotes a finitely generated torsion-free nilpotent group.
	\end{df}

		Let $R$ be a commutative ring. Denote the set of strict upper triangular matrices in $R^{n\times n}$ by $\mathfrak{u}^0(n,R)$. Denote the set of upper unitriangular matrices in $R^{n\times n}$ by $U^1(n,R)$. Note that $\mathfrak{u}^0(n,R)$ is a Lie $R$-algebra with the standard Lie bracket for matrices. The set $U^1(n,R)$ is a nilpotent group for matrix multiplication. Both $\mathcal{I}$-groups and finite-dimensional Lie $\C$-algebras can be realized as subgroups/subalgebras of $\GL(n , \C)$ or $\C^{n\times n}$ for some $n\in \N$ by \cite[Theorem 4 and 5]{Birkhoff} and \cite[Chapter 5, Theorem 2]{sega} respectively:

	\begin{thm} \label{thm_embedding}
		If $L$ is a finite-dimensional Lie algebra over a field $\F$ of characteristic zero, then $L$ can be embedded in $\mathfrak{u}^0(n, \F)$ for some $n\in \N$. If $G$ is an $\mathcal{I}$-group, then $G$ can be embedded in $U^1(n , \Z)$ for some $n\in \N$.
	\end{thm}
	Let $\F$ be a field of characteristic zero. The exponential map for matrices induces a bijection $\exp : \mathfrak{u}^0(n, \F) \to U^1(n , \F)$ with inverse $\log: U^1(n, \F) \to \mathfrak{u}^0(n , \F)$, see \cite[Theorem 1.2.1]{MR1070979}. In fact, this allows one to define a group isomorphism
	\begin{equation} \label{eq_Phi}
	\Phi: U^1(n, \F) \to (\mathfrak{u}^0(n , \F), \ast): M \mapsto \log(M),
	\end{equation}
	using the \emph{Baker-Campbell-Hausdorff formula}
	\begin{equation} \label{eq_BCH}
	M_1\ast M_2 = M_1 + M_2 + \dfrac{1}{2}[M_1,M_2]_L + \sum_{i=3}^n q_i(M_1,M_2).
	\end{equation}
	Here, $[\cdot, \cdot]_L$ denotes the Lie bracket for matrices, and $q_i(M_1,M_2)$ is a specific rational linear combination of nested Lie brackets of length $i$. We refer to \cite[Section 1.2]{MR1070979} for more details.
	
	By Theorem \ref{thm_embedding} and Equation \eqref{eq_Phi}, we know that every $\mathcal{I}$-group $G$ is a subgroup of $(\mathfrak{u}^0(n , \Q), \ast)$ for some $n\in \N$. This observation allows us to give the following definition:
	\begin{df}
		\label{df_completion}
		Let $G \leq (\mathfrak{u}^0(n , \Q), \ast)$ be a finitely generated group and $R$ a commutative ring of characteristic zero containing the constants occuring in equation \eqref{eq_BCH}. The set $G^R = \vct_R G$ consisting of the $R$-linear span of elements in $G$ is called the \emph{$R$-completion} of $G$. The $\Q$-completion and $\C$-completion are also called the \emph{rational and complex Mal'cev completions} of $G$.
	\end{df}
	The $R$-completion $G^R$ is both a group for the Baker-Campbell-Hausdorff formula as in Equation \eqref{eq_BCH} and a Lie $R$-algebra by \cite[Chapter 6]{sega}. The following nilpotent groups will play an important role in this paper:
	\begin{df}
		Let $G \leq (\mathfrak{u}^0(n , \Q), \ast)$ be a finitely generated group. If $G$ is simultaneously a Lie ring, i.e. it is closed under addition and taking the Lie bracket, then we say $G$ is an \emph{LR-group}.
	\end{df}
	
	Let $G\leq (\mathfrak{u}^0(n, \Q), \ast)$ be an LR-group with underlying Lie ring structure $L$. Both $G$ and $L$ have a metric structure, namely as a finitely generated group and as a $\Z$-module, which we introduce below, including the relation between both.
	\begin{df}
		Let $G$ be a finitely generated group with finite generating set $S$. The \emph{word norm} on $G$ via $S$ is defined as 
		$$\norm{g}_{G,S} = \min\{k \mid g = s_1\dots s_k, s_i\in S\cup S^{-1}, k\in \N\cup\{0\}\}.$$
	\end{df} 
	Since $L^\Q$ is a $\Q$-vector space, we can take a norm $\norm{\cdot}_L$ on $L$ by restricting a norm on $L^\Q$. Below, we define the Guivarc'h length of a nilpotent Lie algebra or Lie ring. This length function first appeared in the proof of \cite[Theorem II.1]{MR369608} in order to obtain improved bounds on the word growth of nilpotent groups.
	\begin{df}
		Let $\mathfrak{g}$ denote a finite-dimensional nilpotent Lie $\R$-algebra of nilpotency class $c$ with lower central series $\left(\gamma_i(\mathfrak{g})\right)_{1\leq i\leq c}$. Take vector spaces $\mathfrak{a}_i \subset \mathfrak{g}$ such that $\gamma_i(\mathfrak{g}) = \mathfrak{a}_i \oplus \ldots \oplus \mathfrak{a}_c$ for all $1\leq i \leq c$. Choose a vector space norm $\norm{\cdot}_i$ on $\mathfrak{a}_i$. A \emph{Guivarc'h length} $l_G$ with respect to the chosen decomposition and norms for a vector $v\in \mathfrak{g}$ is given by
		$$l_G(v) := \sup\{(\norm{v_i}_i)^{1/i}\mid v = v_1 + \ldots + v_c \in \mathfrak{a}_1 \oplus \ldots \oplus \mathfrak{a}_c\}.$$
	\end{df}
	\begin{df}
		A Guivarc'h length $l_G$ on a (nilpotent) Lie ring $L$ is defined as the restriction of a Guivarc'h length on $L^\R$ to $L$. 
	\end{df}
	The Guivarc'h length depends on the choice of decomposition $\mathfrak{a}_i$ and the choice of norms $\norm{\cdot}_i$ on each component. However, if $l'_G$ is another Guivarc'h length on the Lie ring $L$, then by \cite{MR369608} there exists a constant $C>0$ such that
	\begin{equation} \label{eq_equiv_Guiv}
		\dfrac{1}{C}l'_G(v) \leq l_G(v) \leq Cl'_G(v).
	\end{equation}
	As $\max\left\{\norm{v_i}_i\mid v = v_1 + \ldots + v_c \in \mathfrak{a}_1 \oplus \ldots \oplus \mathfrak{a}_c\right\}$ defines a genuine norm on $\mathfrak{g}$ and all norms on finite-dimensional vector spaces are equivalent, there exists for every norm $\norm{\cdot}$ on a Lie ring $L$ a constant $C > 0$ such that
	\begin{equation}\label{eq_norm_to_Gui}
		\dfrac{1}{C}\norm{v} \leq l_G(v)^c \leq C\norm{v}^c.
	\end{equation}
	
	In an LR-group, word norms are equivalent with Guivarc'h lengths on the underlying Lie ring structure, as is shown in \cite{MR369608}:
	\begin{prop} \label{prop_Guivarch}
		Let $G\leq (\mathfrak{u}^0(n, \Q), \ast)$ be an LR-group with underlying Lie ring structure $L$. Let $l_G$ denote a Guivarc'h length on $L$ and fix a generating set of $G$ with induced word norm $\norm{\cdot}_G$. Then, there exists a constant $C>0$ such that for all $g\in G$
		$$\dfrac{1}{C}\norm{g}_G \leq l_G(g) \leq C\norm{g}_G .$$
	\end{prop}

	Finally, let us address what happens with automorphisms when passing from an $\mathcal{I}$-group to its rational Mal'cev completion or/and its corresponding Lie algebra.
	\begin{prop} \label{prop_KQ_maps}
		Let $G$ be an $\mathcal{I}$-group with Mal'cev basis $\{g_1, \ldots, g_m\}$, then any automorphism $\varphi:G \to G$ extends uniquely to an automorphism $\varphi^\Q: G^\Q \to G^\Q$. Furthermore, under the identification $G^\Q \cong (\mathfrak{g}, \ast)$, group automorphisms of $G^\Q$ are Lie algebra automorphisms of $\mathfrak{g}$ and vice versa.
	\end{prop}
	\begin{proof}
		See \cite[Theorem 9.20]{Khukhro} and \cite[Theorem 10.13(f)]{Khukhro} respectively.
	\end{proof}
	\subsection{Residual Finiteness Growth} \label{sec_RF}
	In this subsection, we will introduce the residual finiteness growth of finitely generated groups and of Lie rings. The notion for groups originally appeared in \cite{bou2010quantifying} and has produced a large number of results as mentioned in the introduction. The notion for Lie rings is however new, although other generalization of residual finiteness growth to certain algebraic structures exist, see for example \cite{bou2010quantifying,franz2017quantifying}, where residual finiteness growth is defined for rings of integers over number fields and finitely generated integral domains respectively.

		Throughout this subsection, let $G$ denote a finitely generated residually finite group with finite generating set $S$. Let $L$ denote a Lie ring and $\norm{\cdot}_L$ a norm or a Guivarc'h length on $L$. We will assume that $\{v\in L\mid \norm{v}_L\leq 1\} \neq \{0\}$.

	\begin{df}
		Let $\mathcal{P}_G$ and $\mathcal{P}_L$ denote subsets of normal subgroups $N\lhd G$ and of ideals $I\lhd L$ respectively. The \emph{divisibility functions} $D_{G, \mathcal{P}_G}: G\setminus\{e\} \to \N\cup\{\infty\}$ and $D_{L, \mathcal{P}_L}: L\setminus\{0\} \to \N\cup\{\infty\}$ are defined by
		\begin{equation*}
			\begin{split}
				D_{G, \mathcal{P}_G}(g) & = \min\{[G:N]\mid g\notin N \lhd G, N\in \mathcal{P}_G\}, \\
				D_{L, \mathcal{P}_L}(v) & = \min\{[L:I]\mid v\notin I \lhd L, I \in \mathcal{P}_L\}
			\end{split}
		\end{equation*}
		where $\min \emptyset = \infty$. If $\mathcal{P}_G$ and $\mathcal{P}_L$ are the sets of all normal subgroups and ideals respectively, then we simply write $D_G$ and $D_L$.
	\end{df}
\noindent	By the definition of residual finiteness, $D_G(g) < \infty$ for every non-trivial element $g\in G$. Given $v\neq 0$, one can use the ideal $mL$ for some $m\in \N$ sufficiently large to show that also $D_L(v) < \infty$.
	
	The residual finiteness growth estimates the divisibility function in terms of the distance of $g$ or $v$ to the origin/neutral element.
	\begin{df}
		The \emph{residual finiteness growth} $\RF_{G, \mathcal{P}_G,S}: \R_{\geq 1} \to \N\cup\{\infty\}$ and $\RF_{L, \mathcal{P}_L, \norm{\cdot}_L}: \R_{\geq 1} \to \N\cup \{\infty\}$ are defined by
		\begin{equation*}
			\begin{split}
				\RF_{G,\mathcal{P}_G,S}(r) & = \max\{D_{G,\mathcal{P}_G}(g)\mid 0<\norm{g}_{G,S} \leq r\}, \\
				\RF_{L,\mathcal{P}_L, \norm{\cdot}_L}(r) & = \max\{D_{L,\mathcal{P}_L}(v)\mid 0<\norm{v}_L\leq r\}.
			\end{split}
		\end{equation*}
	\end{df}
	If we had chosen another finite generating set $T$ of $G$, then there was a constant $C>0$ such that
	$$\dfrac{1}{C}\norm{g}_{G,T} \leq \norm{g}_{G,S} \leq C\norm{g}_{G,T},$$
	therefore,
	$$\RF_{G,T}\left(\dfrac{1}{C} r \right) \leq \RF_{G,S}(r) \leq \RF_{G,T}\left(Cr\right).$$
	Now, we obtain a definition of residual finiteness growth of a group, written as $\RF_{G, \mathcal{P}_G}$ or $\RF_G$ for $D_G$, independent of the choice of generating set if we consider the residual finiteness growth up to the following equivalence relation:
	\begin{df}
		Let $f,g: \mathbb{R}_{\geq 1} \to \mathbb{R}_{\geq 1}$ be non-decreasing functions. We write 
		\begin{align*}f &\preceq g \Leftrightarrow \exists C >0: \forall r \geq \max\{1,1/C\}: f(r) \leq Cg(Cr);\\
			f&\approx g \Leftrightarrow f\preceq g \text{ and } g \preceq f.\end{align*}
	\end{df}
	By the equivalence of norms on finite-dimensional vector spaces, also the residual finiteness growth of $L$ is well-defined up to this equivalence relation. We denote this invariant by $\RF_{L, \mathcal{P}_L}^{\nom}$, or simply $\RF_{L}^{\nom}$ if we take $D_L$. By Equation \eqref{eq_equiv_Guiv}, we similarly obtain $\RF_{L, \mathcal{P}_L}^{\Gui}$ and $\RF_{L}^{\Gui}$ for Guivarc'h lengths. We have the following relation:
	\begin{lemma} \label{lem_link_Gui_norm}
		Let $L$ be a Lie ring of nilpotency class $c$, and let $\mathcal{P}_L$ denote a subset of ideals of $L$. We have
		\begin{equation*}
			\RF_{L, \mathcal{P}_L}^{\nom}(r) \preceq \RF_{L, \mathcal{P}_L}^{\Gui}(r) \preceq \RF_{L, \mathcal{P}_L}^{\nom}(r^c) .
		\end{equation*}
	\end{lemma}
	\begin{proof}
		This is a direct consequence of Equation \eqref{eq_norm_to_Gui}. 
	\end{proof}
	Since polylogarithmic functions satisfy $\log^{k}(r^c) \approx \log^{k}(r)$ for any $k\in\N$, this lemma will lead to $\RF_L^{\Gui} = \RF_L^{\nom}$ after proving our main result. There is one property of residual finiteness growth that will play an instrumental role throughout this article.
	\begin{df}
		Let $H$ and $G$ be groups. Let $\varphi: H \to \Out(G)$ be a homomorphism. We say a normal subgroup $N$ of $G$ is $H$-invariant or $\varphi$-invariant if $\varphi(h)(N) = N$ for all $h\in H$.
	\end{df}
\begin{rem}
	Note that $\varphi(h)(N)$ is well-defined. Indeed, since $N$ is normal, $\psi(N) = N$ for any inner automorphism, so $\psi_1(N) = \psi_2(N)$ for any $\psi_1$ and $\psi_2$ in the set of automorphisms defined by $\varphi(h)$. 
\end{rem}
Let $1 \to G \to \Gamma \to H \to 1$ be a short exact sequence with projection $\pi: \Gamma \to H$. Given $h\in H$ and $g\in \Gamma$ with $\pi(g) = h$, then 
$$ \psi_g : G \to G : k \mapsto gkg^{-1}$$
defines an automorphism of $G$. The map 
$$\varphi: H \to \Out(G) : h \mapsto  \psi_g\Inn(G) \text{ with }\pi(g) = h$$
gives a well-defined homomorphism. 
\begin{prop} \label{prop_RF_fin_ext_Hinv}
	Let $\Gamma$ be a residually finite, finitely generated group and $1 \to G \to \Gamma \to H \to 1$ be a short exact sequence with $|H| < \infty$. Let $\mathcal{P}_{H\curvearrowright G}$ denote the set of $H$-invariant normal subgroups in $G$. Then, $\RF_{\Gamma} = \RF_{G, \mathcal{P}_{H\curvearrowright G}}$.
\end{prop}
\begin{proof}
	See \cite[Theorem 3.3]{math_virt_ab}.
\end{proof}
In our case, $\Gamma$ will be virtually nilpotent and thus we can choose $G$ to be an $\mathcal{I}$-group, and obtain $\RF_{\Gamma} = \RF_{G, \mathcal{P}_{H\curvearrowright G}}$. The calculation of $\RF_{G, \mathcal{P}_{H\curvearrowright G}}$ will be the focus of the rest of the article, where we only assume that $H$ is a finitely generated subgroup of $\Aut(G)$. By taking 
$$\varphi(H) = \dfrac{\tilde{H}}{\Inn(G)},$$
for some $\tilde{H} \le \Aut(G)$, we observe that $H$-invariance and $\tilde{H}$-invariance are formally the same. Furthermore, $\tilde{H} \subset \Aut(G)$ is finitely generated as $H$ is finite and $\Inn(G)$ is finitely generated.

\section{Commensurable nilpotent groups} \label{sec_invariance}
In this section, we will prove Theorem \ref{thm_intro_Q_inv2}. 
When $\RF_{G, \mathcal{P}_{H\curvearrowright G}}$ is given with $G$ an $\mathcal{I}$-group, this result will allow us to replace $G$ by another $H$-invariant $\mathcal{I}$-group. In the next section, we will replace $G$ with an $H$-invariant LR-group to establish a connection with the residual finiteness growth in Lie rings. As a direct consequence of Theorem \ref{thm_intro_Q_inv2}, we will also argue that commensurable nilpotent groups have the same residual finiteness growth. We begin by recalling two different constructions of normal subgroups.
\begin{df}
	Let $H \subset G$ be a subgroup.
	\begin{itemize}
		\item The core of $H$ in $G$, denoted by $\core_G(H)$, is the largest normal subgroup of $G$ contained in $H$. It is given by $$\core_G(H) := \bigcap_{g\in G} gHg^{-1}.$$
		\item The normalizer of $H$ in $G$, denoted by $N_G(H)$, is the largest subgroup of $G$ such that $H$ is normal in it. It is given by $$N_G(H) = \{g\in G\mid gHg^{-1} = H\}.$$
	\end{itemize}
\end{df}
\noindent	Note that if $H\leq_f G$, then $\core_G(H) \lhd_f G$. The following statement given in \cite[Lemma 1.7]{grunewald1988subgroups} is crucial in the proof.
\begin{lemma}
	Let $G$ be an $\mathcal{I}$-group of nilpotency class $c$ and Hirsch length $h$. If $H$ is a subgroup of $G$ such that $[G: N_G(H)] \mid p$ for some prime $p$, then $[H: \core_G(H)] \mid p^{hc(c-1)/2}$.
\end{lemma}
This has the following direct consequence.
\begin{lemma}
	Let $G_1$ and $G_2$ be two $\mathcal{I}$-groups. If $G_1 \lhd G_2$ with $[G_2:G_1] = p$, then there exists a constant $C>0$ such that $[G_2: \core_{G_2}(N)] \leq C [G_1:N]$ for all $N\lhd_f G_1$.
\end{lemma}
\begin{proof}
	We will argue that $[G_2: \core_{G_2}(N)] \leq p^{1+hc(c-1)/2}[G_1:N]$, where $h$ and $c$ are the Hirsch and step length of $G_2$ respectively. Since $N_{G_2}(N)$ is the largest group $H$ such that $N\lhd H \leq G_2$ and $N\lhd G_1$, it holds that $G_1 \leq N_{G_2}(N)$. In particular, 
	$$[G_2: N_{G_2}(N)] \mid [G_2: G_1] \mid p$$
	and thus the previous lemma implies that $[N: \core_{G_2}(N)] \leq p^{hc(c-1)/2}$. Hence, as claimed,
	$$[G_2: \core_{G_2}(N)] \leq [G_2: G_1] [G_1: N] [N: \core_{G_2}(N)] \leq p^{1+hc(c-1)/2}[G_1:N].$$
\end{proof}
Applying this lemma inductively, we obtain the result below.
\begin{prop}
	Let $G_1 \lhd_f G_2$ be two $\mathcal{I}$-groups. There exists a constant $C>0$ such that $[G_2: \core_{G_2}(N)] \leq C [G_1:N]$ for all $N\lhd_f G_1$.
\end{prop}
\begin{proof}
	First, we claim that there exists a chain of normal subgroups
	$$G_1 = H_k \lhd H_{k-1} \lhd \ldots \lhd H_1 \lhd H_0 = G_2,$$
	such that $[H_{i-1}:H_{i}]$ is prime for all $i \in \{1, 2, \ldots ,k\}$. 		
	
	In order to show the claim, set $H_0 = G_2$. Since $H_0$ is nilpotent and $G_1 \lhd_f H_0$, the quotient group $H_0/G_1$ exists and is a finite nilpotent group. Hence, it is a direct sum of Sylow $p$-subgroups $P_1 \oplus P_2 \oplus \ldots \oplus P_n$. Take any maximal subgroup $N$ in $P_1$, then \cite[Lemma 4.4(a)]{Khukhro} implies that $N$ is normal in $P_1$ and $[P_1: N]$ is prime. Now, we define $H_1$ as the inverse image of $N\oplus P_2\oplus \ldots \oplus P_n$ under the projection $H_0 \to H_0/G_1$. It is clear that $H_1 \lhd H_0$ and $[H_0:H_1]$ is prime. We can now exploit that $G_1 \lhd_f H_1$ is a normal subgroup of smaller index to inductively construct the desired chain of normal subgroups, proving the claim.
	
	Now, take $N \lhd_f G_1$. We know that
	$$[H_{k-1}: \core_{H_{k-1}}(N)] \leq C_{k-1} [G_1:N]$$
	for some $C_{k-1}>0$ depending on $G_1$ and $H_{k-1}$ but independent of $N$. We also see that
	$$[H_{k-2}: \core_{H_{k-2}}(\core_{H_{k-1}}(N))] \leq C_{k-2} [H_{k-1}: \core_{H_{k-1}}(N)] \leq C_{k-2}C_{k-1} [G_1:N].$$
	Since $\core_{H_{k-2}}(\core_{H_{k-1}}(N))$ is a normal subgroup of $H_{k-2}$ contained in $N$, we know by definition that $\core_{H_{k-2}}(\core_{H_{k-1}}(N)) \leq \core_{H_{k-2}}(N),$ which implies
	$$[H_{k-2}: \core_{H_{k-2}}(N)] \leq [H_{k-2}: \core_{H_{k-2}}(\core_{H_{k-1}}(N))] \leq C_{k-2}C_{k-1} [G_1:N].$$
	Now proceed inductively to conclude.
\end{proof}
\begin{proof}[Proof of Theorem \ref{thm_intro_Q_inv2}]
	Suppose $G_1 \lhd_f G_2$ are given $\mathcal{I}$-groups. We must prove that $\RF_{G_1, \mathcal{P}_{H\curvearrowright G_1}} = \RF_{G_2, \mathcal{P}_{H\curvearrowright G_2}}$, where $H\leq \Aut(G_2)$ and where $G_1$ is $H$-invariant.	To show this, fix finite generating sets $S$ and $T$ of $G_1$ and $G_2$ respectively. We may assume that $S \subset T$. Since $G_1$ has finite index in $G_2$, we know that there exists a constant $C>0$ such that $\norm{g}_{G_1,S} \leq C \norm{g}_{G_2,T}$ for all $g\in G_1$ by \cite[Corollary 5.4.5]{loh2017geometric}.
	
	First, let us show that $\RF_{G_1, \mathcal{P}_{H\curvearrowright G_1}} \preceq \RF_{G_2, \mathcal{P}_{H\curvearrowright G_2}}$. Take $e\neq g_1 \in G_1$ with $\norm{g_1}_{G_1,S} \leq r$ arbitrary. Since $S \subset T$, we see that $\norm{g_1}_{G_2,T} \leq r$. By definition of $\RF_{G_2, \mathcal{P}_{H\curvearrowright G_2}}$, we can find an $H$-invariant normal subgroup $N$ such that $g_1 \notin N$ and
	$$[G_2: N] \leq  \RF_{G_2, \mathcal{P}_{H\curvearrowright G_2}}(r).$$
	Now, $g_1 \notin N\cap G_1 \lhd G_1$. Furthermore, $N\cap G_1$ is $H$-invariant as the intersection of $H$-invariant subgroups. We thus observe that
	$$D_{G_1, \mathcal{P}_{H \curvearrowright G_1}}(g_1) \leq [G_1:N\cap G_1] \leq [G_2:N] \leq \RF_{G_2, \mathcal{P}_{H\curvearrowright G_2}}(r).$$
	Conclude by taking the maximum over all elements $g_1 \in G_1$ with $0 < \norm{g_1}_{G_1,S} \leq r$.
	
	We will proceed to show that $\RF_{G_1, \mathcal{P}_{H\curvearrowright G_1}} \succeq \RF_{G_2, \mathcal{P}_{H\curvearrowright G_2}}$. For this, take $g_2\in G_2$ with $0 < \norm{g_2}_{G_2,T} \leq r$ arbitrarily. If $g_2\notin G_1$, then $D_{G_2, \mathcal{P}_{H \curvearrowright G_2}}(g_2) \leq [G_2:G_1]$, so from now on we assume that $g_2\in G_1$. In particular, $\norm{g_2}_{G_1,S}\leq Cr$. By definition, there exists an $H$-invariant normal subgroup $N$ of $G_1$ such that $g_2 \notin N$ and 
	$$[G_1: N] \leq \RF_{G_1, \mathcal{P}_{H\curvearrowright G_1}}(Cr).$$
	We claim that $\core_{G_2}(N)$ is $H$-invariant. Indeed, let $\xi \in H$, then
	$$\xi(\core_{G_2}(N)) = \bigcap_{g\in G_2} \xi(g)N\xi(g)^{-1} = \core_{G_2}(N),$$
	since $N$ is $H$-invariant. As $g_2 \notin \core_{G_2}(N) \leq N$, we obtain
	$$D_{G_2, \mathcal{P}_{H\curvearrowright G_2}}(g_2) \leq [G_2: \core_{G_2}(N)] \leq C_2[G_1: N] \leq C_2 \RF_{G_1, \mathcal{P}_{H\curvearrowright G_1}}(Cr)$$
	for a fixed $C_2>0$ by the previous result. 
	Taking the maximum of $D_{G_2, \mathcal{P}_{H\curvearrowright G_2}}(g_2)$ over all $g_2\in G_2$ with $0<\norm{g_2}_{G_2,T} \leq r$ shows that
	$$\RF_{G_2}(r) \leq \max\{[G_2:G_1],  C_2  \RF_{G_1}(Cr)\} \preceq \RF_{G_1}(r),$$
	ending the proof.
\end{proof}
Specifically for nilpotent groups, this result implies that the residual finiteness growth is a commensurability invariant for them.
\begin{df}
	Two groups $G_1$ and $G_2$ are called \emph{commensurable} if there exist subgroups $H_1 \leq_f G_1$ and $H_2 \leq_f G_2$ such that $H_1\cong H_2$ are isomorphic.
\end{df}
\begin{thm} \label{thm_comm}
	Let $G_1$ and $G_2$ denote two finitely generated nilpotent groups. If $G_1$ and $G_2$ are commensurable, then $\RF_{G_1} = \RF_{G_2}$.
\end{thm}
Let us first show the following well-known result, see e.g. \cite[Theorem 2.1]{MR283082}, which we reprove for the convenience of the reader. 
\begin{lemma}\label{lem_baumslag}
	Let $G$ be a finitely generated nilpotent group with torsion subgroup $T$. There exists an embedding $i: G \to G/T\times Q$ for some finite group $Q$ such that $i(G) \leq_f G/T\times Q$.
\end{lemma}
\begin{proof}
	Since $T$ is finite and $G$ is residually finite, we can find a homomorphism $\psi: G \to Q$ to a finite group $Q$, such that $\psi(t) \neq e$ for all non-trivial $t\in T$. Now, the map $i$ can be taken as
	$$i: G \to G/T \times Q: g \mapsto (gT, \psi(g)).$$
	If $\pi$ denotes the projection of $G/T\times Q$ onto $G/T$, then $\pi(i(G)) = G/T$, so
	$$[G/T\times Q: i(G)] = [G/T: \pi(i(G))]\cdot [Q: Q\cap i(G)] \leq 1 \cdot |Q| < \infty. $$
\end{proof}
\begin{proof}[Proof of Theorem \ref{thm_comm}]
	We first show that for general nilpotent groups $G_1 \leq_f G_2$ implies $\RF_{G_1} = \RF_{G_2}$. Let $T$ denote the torsion subgroup of $G_2$. Consider the embedding $i: G_2 \to G_3\times Q$ with $G_3 = G_2/T$ of Lemma \ref{lem_baumslag}. Identify $G_2$ with its embedding $i(G_2)$ in $G_3\times Q$. Now, we have inclusions
	$$G_1 \leq_f G_2 \leq_f G_3\times Q.$$
	The intersection $G = \core_{G_3}(G_1\cap G_3)$ is torsion-free as a subgroup of the torsion-free group $G_3$, and clearly of finite index in $G_1$.
	Hence, $G$ is also a finite index normal subgroup of $G_3$.
	In particular, Theorem \ref{thm_intro_Q_inv2} (with $H$ trivial) states that $\RF_G = \RF_{G_3}$. By all the inclusions above, we have that
	$$\RF_G \preceq \RF_{G_1} \preceq \RF_{G_2} \preceq \RF_{G_3\times Q} = \max\{\RF_{G_3}, \RF_{Q}\} = \RF_{G_3}.$$
	Thus, $\RF_{G_1} = \RF_{G_2}$ follows immediately from $\RF_G = \RF_{G_3}$.
	
	Now assume that $G_1$ and $G_2$ are commensurable nilpotent groups. Take groups $H_1$ and $H_2$ such that $H_1 \leq_f G_1$ and $H_2 \leq_f G_2$ with $H_1\cong H_2$. By subgroup inclusions, we have
	$$\RF_{G_1} \succeq \RF_{H_1} = \RF_{H_2} \preceq \RF_{G_2}.$$ Applying the previous statement to the equation above implies that $\RF_{G_1} = \RF_{G_2}$ as desired.
\end{proof}
We have the following two immediate consequences, where the first thus gives an alternative proof for the one provided in \cite{MarkPubl}. 
\begin{cor}
	Let $G$ be a finitely generated nilpotent group with torsion subgroup $T$, then $\RF_G = \RF_{G/T}$.
\end{cor}
\begin{proof}
	This follows from Lemma \ref{lem_baumslag}.
\end{proof}
\begin{cor} \label{cor_from_I_to_LR}
	Let $G_1$ be a finitely generated nilpotent group, then there exists an LR-group $G_2$ such that $\RF_{G_1} = \RF_{G_2}$.
\end{cor}
\begin{proof}
	The previous corollary says that $\RF_{G_1}$ equals $\RF_{G_3}$ for some $\mathcal{I}$-group $G_3$. By \cite[Chapter 6, Part B]{sega}, we know that every $\mathcal{I}$-group is commensurable with some LR-group.
\end{proof}

\section{Residual finiteness growth of nilpotent Lie rings} \label{sec_Lie_inv}
	By Proposition \ref{prop_RF_fin_ext_Hinv}, we know that $\RF_{\Gamma}$ for a virtually nilpotent group equals $\RF_{G_1, \mathcal{P}_{H\curvearrowright G_1}}$ for some $\mathcal{I}$-group $G_1$ and finitely generated $H\leq \Aut(G_1)$. Furthermore, by Theorem \ref{thm_intro_Q_inv2}, we know that $\RF_{G_1, \mathcal{P}_{H\curvearrowright G_1}}$ equals $\RF_{G_2, \mathcal{P}_{H\curvearrowright G_2}}$ if $G_2 \lhd_f G_1$ is $H$-invariant. This section will allow us on the one hand to  assume that $G_2$ is an $H$-invariant LR-group and on the other hand translate $\RF_{G_2, \mathcal{P}_{H\curvearrowright G_2}}$ to an equivalent function over its underlying Lie ring structure as in Theorem \ref{thm_intro_RFG_RFL}. These statements will be proven in the third part, using the technicalities of the first two parts.

	\subsection{From Ideal to Normal Subgroup} \label{sec_ideal_to_normal}
	In this section, we will focus on proving Proposition \ref{prop_ideal_to_normal} as a tool for Theorem \ref{thm_intro_RFG_RFL}.

	\begin{prop} \label{prop_ideal_to_normal}
		Let $G\leq (\mathfrak{u}^0(n, \Q), \ast)$ be an LR-group with underlying Lie ring structure $L$. There exists a constant $M_1>0$ such that if $I \lhd_f L$, then there exists a normal subgroup $N \subset I$ of $G$ such that $[G:N] \leq M_1[L:I]$. Furthermore, if $H \leq \Aut(L^\Q)$ and $L$ and $I$ are $H$-invariant, then $N$ can be chosen $H$-invariant.
	\end{prop}
	\begin{rem}
		The normal subgroup $N$ provided by the proof of this proposition will be an LR-group.
	\end{rem}
	This result will be derived from the Baker-Campbell-Hausdorff formula.
	\begin{df}
		Let $X \subset L$ be a finite subset, then the basic Lie bracket $q$ of length $1$ are exactly the elements of $X$. Inductively, the basic Lie bracket $q$ of length $l \geq 2$ in $X$ are those of the form $q =  [q_1,q_2]_L$ where $q_1$ and $q_2$ are basic Lie brackets of length $l_1$ and $l_2$ with $l = l_1 + l_2$. 
	\end{df}
	\begin{nota}
		Let $G\leq (\mathfrak{u}^0(n, \Q), \ast)$ be an LR-group with underlying Lie ring structure $L$. Recall that the Baker-Campbell-Hausdorff formula dictates that
		\begin{equation} \label{eq_BCH_1}
			v\ast w = v + w + \dfrac{1}{2}[v,w]_L + \sum_{i=3}^n q_i(v,w)
		\end{equation}
		and
		\begin{equation} \label{eq_BCH_2}
			[v,w]_G = [v,w]_L + \sum_{i=3}^n \tilde{q}_i(v,w),
		\end{equation}
		where $q_i(v,w)$ and $\tilde{q}_i(v,w)$ are specific rational linear combinations of basic Lie brackets of length $i$.
	\end{nota}
	We start by defining the subalgebra on which we will be working:
	\begin{df} \label{df_IG}
		Let $L\leq \mathfrak{u}^0(n, \Q)$ be a Lie ring. Fix $\Delta \in \N$ such that the rational coefficients of the Baker-Campbell-Hausdorff formulae in Equations \eqref{eq_BCH_1}-\eqref{eq_BCH_2} lie in $(1/\Delta)\Z$. We define for every ideal $I \triangleleft L$ the set
		
		$$\mathcal{G}(L,I, \Delta) = \vct_\Z\{\Delta^{c-i+1}(\gamma_i(L^\Q)\cap I)\mid 1\leq i \leq c\}.$$
	\end{df}
	\begin{lemma} \label{lem_q_internal}
		Let $v, w \in L$, and let $q(v,w)$ be a basic Lie bracket of length at least $2$ in $\{v,w\}$. If $v\in \mathcal{G}(L,I, \Delta)$, then $q(v,w) \in \Delta \mathcal{G}(L,I, \Delta)$.
	\end{lemma}
	\begin{proof}
		Take a general element of $\mathcal{G}(L,I, \Delta)$, which is given by a $\Z$-linear combination $\lambda_1v_1 + \ldots + \lambda_{k-1}v_{k-1}$, where every $v_{j}$ is of the form $\Delta^{c-i+1}w_j$ with $w_j \in \gamma_i(L^\Q)\cap I$ for some $i$, and an arbitrary element $v_k \in L$. By linearity of the Lie bracket,
		$$ q(\lambda_1v_1 + \ldots + \lambda_{k-1}v_{k-1}, v_k)$$
		can be rewritten as a $\Z$-linear combination of basic Lie brackets of length at least $2$ in $\{v_1, \ldots , v_k\}$. It suffices to show that all these basic brackets lie in $\Delta \mathcal{G}(L,I, \Delta)$ as $\mathcal{G}(L,I, \Delta)$ is additively closed by construction.
		
		Let $\tilde{q}(v_1, \ldots , v_k)$ be any basic Lie bracket of length at least $2$ as above. Without loss of generality we can assume that $v_1$ appears in it. Write $v_1 = \Delta^{c-i+1}w_1$ with $w_1 \in \gamma_i(L^\Q)\cap I$ for some $i$. Since $I$ is an ideal in $L$ and the Lie bracket $\tilde{q}$ is basic, we have that 
		$$\tilde{q}(w_1, v_2, \ldots, v_k) \in \gamma_{i+1}(L^\Q)\cap I.$$
		Using linearity, 
		\begin{equation*}
			\begin{split}
				\tilde{q}(\Delta^{c-i+1}w_1, v_2, \ldots, v_k) & = \Delta^{l(c-i+1)}\tilde{q}(w_1, v_2, \ldots, v_k) \\ 
				& = \Delta (\Delta^{l(c-i+1)-1}\tilde{q}(w_1, v_2, \ldots, v_k))
			\end{split}
		\end{equation*}
		for some $l\in\N$. As $l(c-i+1)-1 \geq c-i$, we conclude that $\tilde{q}(v_1, \ldots , v_k) \in \Delta \mathcal{G}(L,I, \Delta)$.
	\end{proof}
	\begin{cor} \label{cor_IG_is_normal}
		The set $\mathcal{G}(L,I, \Delta)$ is an ideal of $L$ included in $I$. Furthermore, if $H \leq \Aut(L^\Q)$ and $L$ and $I$ are $H$-invariant, then $\mathcal{G}(L,I, \Delta)$ is $H$-invariant.
	\end{cor}
	\begin{proof}
		The set $\mathcal{G}(L,I, \Delta)$ is defined as a $\Z$-linear span, so it is surely a $\Z$-module. We see that $\mathcal{G}(L,I, \Delta)$ is an ideal of $L$, since $[\mathcal{G}(L,I, \Delta), L]_L \subset \mathcal{G}(L,I, \Delta)$ by Lemma \ref{lem_q_internal}. The inclusion $\mathcal{G}(L,I, \Delta) \subset I$ follows directly from the definition of $\mathcal{G}(L,I, \Delta)$.
		
		Let $\xi \in H$. We know that $\xi(L) = L$ and $\xi(I) = I$. Also, $\xi(\gamma_i(L^\Q)) = \gamma_i(L^\Q)$. We see that
		$\xi(\Delta^{c-i+1}I\cap \gamma_i(L^\Q)) = \Delta^{c+1-i}I\cap \gamma_i(L^\Q)$ and thus $\xi(\mathcal{G}(L,I, \Delta)) = \mathcal{G}(L,I, \Delta)$ by linearity.
	\end{proof}
	\begin{lemma}\label{lem_GLI_normal}
		The ideal $\mathcal{G}(L,I, \Delta)$ defines an LR-group. Furthermore, if $L$ is an LR-group itself, then $\mathcal{G}(L,I, \Delta)$ is normal in $L$.
	\end{lemma}
	\begin{proof}
		Let $q_i$ and $\tilde{q}_i$ be defined as in Equations \eqref{eq_BCH_1} and \eqref{eq_BCH_2}. The $\Q$-linear combinations in $q_i$ and $\tilde{q}_i$ have coefficients in $(1/\Delta)\Z$.
		By Lemma \ref{lem_q_internal}, we now see that
		$$q_i(\mathcal{G}(L,I, \Delta), \mathcal{G}(L,I, \Delta)) \in \mathcal{G}(L,I, \Delta) \text{ and } \tilde{q}_i(\mathcal{G}(L,I, \Delta), L) \in \mathcal{G}(L,I, \Delta).$$
		Thus, $\mathcal{G}(L,I, \Delta)\ast \mathcal{G}(L,I, \Delta) \subset \mathcal{G}(L,I, \Delta)$ and $[\mathcal{G}(L,I, \Delta), L]_G \subset \mathcal{G}(L,I, \Delta)$ by Equations \eqref{eq_BCH_1}-\eqref{eq_BCH_2}. This shows that $\mathcal{G}(L,I, \Delta)$ is an LR-group. If $L$ is also a group, then the latter implies that $\mathcal{G}(L,I, \Delta)$ is normal in it.
	\end{proof}
	\begin{lemma}
		Let $G\leq (\mathfrak{u}^0(n, \Q), \ast)$ be an LR-group with underlying Lie ring structure $L$. If $I$ is an ideal of $L$, then the index of $\mathcal{G}(L,I, \Delta)$ as an ideal of $L$ and as a subgroup of $G$ agree. In particular, $[G:\mathcal{G}(L,I, \Delta)] \leq \Delta^{c\dim_\Q L^\Q}[L:I]$.
	\end{lemma}
	\begin{proof}
		Let $v\in L$. We claim that $v+ \mathcal{G}(L,I, \Delta) = v\ast \mathcal{G}(L,I, \Delta)$. First, we will demonstrate that $v\ast \mathcal{G}(L,I, \Delta) \subset v + \mathcal{G}(L,I, \Delta)$. Take $w\in \mathcal{G}(L,I, \Delta)$. Now,
		$$v\ast w =  v + w + \sum_{i=2}^n q_i(v,w).$$
		By Lemma \ref{lem_q_internal}, we know that $q_i(L,\mathcal{G}(L,I, \Delta)) \subset \mathcal{G}(L,I, \Delta)$, so $w + \sum_{i=2}^n q_i(v,w) \in \mathcal{G}(L,I, \Delta)$. Hence, $v\ast w \in v + \mathcal{G}(L,I, \Delta)$. This shows the first inclusion.
		
		For the other inclusion, we will show that for every $w\in \mathcal{G}(L,I, \Delta)$ and every $v\in L$ there exists $w'\in \mathcal{G}(L,I, \Delta)$ such that $v+w = v\ast w'$. We proceed by induction on $i$ where $w\in \mathcal{G}(L,I, \Delta)\cap \gamma_i(L^\Q)$. If $i = c$, then 
		$$v + w =  v\ast w - \sum_{i=2}^n q_i(v,w) = v\ast w,$$
		since $q_i(v,w) = 0$ by the nilpotency class of $L^\Q$. Hence, the claim holds for $i = c$ with $w' = w$, for every choice of $v\in L$. Now suppose the claim holds for all $w\in \mathcal{G}(L,I, \Delta)\cap \gamma_i(L^\Q)$. Take $w_2\in \mathcal{G}(L,I, \Delta)\cap \gamma_{i-1}(L^\Q)$. Now,
		$$v + w_2 =  v\ast w_2 - \sum_{i=2}^n q_i(v,w_2),$$
		where $w_3 = -\sum_{i=2}^n q_i(v,w_2)$ lies in $\mathcal{G}(L,I, \Delta)\cap \gamma_i(L^\Q)$. Now, $(v\ast w_2) + w_3 = (v\ast w_2) \ast w_4$ for some $w_4\in \mathcal{G}(L,I, \Delta)$ by the induction hypothesis. We conclude that $v+w_2 = v\ast w_2'$ with $w_2' = w_2\ast w_4 \in \mathcal{G}(L,I, \Delta)$. By the principal of induction, we conclude that $v+\mathcal{G}(L,I, \Delta) \subset v \ast \mathcal{G}(L,I, \Delta)$.
		
		The fact that the cosets in the Lie ring setting and in the group setting agree, $v+ \mathcal{G}(L,I, \Delta) = v \ast \mathcal{G}(L,I, \Delta)$, implies that $[L:\mathcal{G}(L,I, \Delta)] = [G:\mathcal{G}(L,I, \Delta)]$. By construction, we also have $\Delta^c I \leq \mathcal{G}(L,I, \Delta) \leq I$, and thus $[L: \mathcal{G}(L,I, \Delta)] \leq \Delta^{c\dim_\Q L^\Q}[L:I]$.
	\end{proof}
	This lemma finishes the proof of Proposition \ref{prop_ideal_to_normal} by setting $N = \mathcal{G}(L, I, \Delta)$.
	\subsection{From Normal Subgroup to Ideal} \label{sec_normal_to_ideal}
	In this subsection, we will prove the analog of Proposition \ref{prop_ideal_to_normal}, given below in Proposition \ref{prop_normal_to_ideal}. This result will be used in the proofs of Theorems \ref{thm_from_I_to_LR} and \ref{thm_intro_RFG_RFL} below.
	\begin{prop} \label{prop_normal_to_ideal}
		Let $G\leq (\mathfrak{u}^0(n, \Q), \ast)$ be an LR-group with underlying Lie ring structure $L$. There exist constants $M_2 >0$ such that if $N \lhd_f G$, then there exists an ideal $I \subset N$ of $L$ such that $[L:I] \leq M_2[G:N]$. Furthermore, if $H \leq \Aut(G^\Q)$ and $G$ and $N$ are $H$-invariant, then $I$ can be chosen $H$-invariant.
	\end{prop}
	The proof of this result is essentially analogous to the proof of Proposition \ref{prop_ideal_to_normal}. However, some arguments are more involved, since the group commutator is not `linear' as is the case for the Lie bracket.
	\begin{df}
		Let $X \subset G$ be a finite subset, then the basic commutators $\kappa$ of length $1$ in $X$ are exactly the elements of $X$. Inductively, the basic commutators $\kappa$ of length $l$ in $X$ are of the form $\kappa = [\kappa_1, \kappa_2]_G$ where $\kappa_1$ and $\kappa_2$ are basic commutators of length $l_1$ and $l_2$ with $l = l_1 + l_2$. 
	\end{df}
	\begin{nota}
		Let $G\leq (\mathfrak{u}^0(n, \Q), \ast)$ be an LR-group with underlying Lie ring structure $L$. The inverse Baker-Campbell-Hausdorff formulas for nilpotent groups (see \cite[Lemma 10.7]{Khukhro}) dictate that there exist numbers $N_1, N_2 \in \N$ such that
		\begin{equation} \label{eq_BCH_inv_1}
			g + h = g\ast h \ast \prod_{j=1}^{N_1} (\kappa_j(g,h))^{r_j}
		\end{equation}
		and
		\begin{equation} \label{eq_BCH_inv_2}
			[g,h]_L = [g,h]_G\ast \prod_{j=1}^{N_2} (\tilde{\kappa}_j(g,h))^{s_j},
		\end{equation}
		where $\kappa_j(g,h)$ and $\tilde{\kappa}_j(g,h)$ are basic group commutators in $\{g,h\}$ of length at least $2$ and $3$ respectively and $r_j, s_j \in \Q$. Fix $\Lambda \in \N$ such that all $r_j, s_j \in (1/\Lambda)\Z$.
	\end{nota}
	Note that these products are finite, because there are only finitely many non-trivial basic commutators in $\{g,h\}$ as $\langle g,h\rangle$ is nilpotent. We may assume that every commutator contains both $g$ and $h$, otherwise it would surely be trivial.
	\begin{df} \label{df_LGN}
		Let $G\leq (\mathfrak{u}^0(n, \Q), \ast)$ be an LR-group of nilpotency class $c$ with underlying Lie ring structure $L$. Let $N$ be a normal subgroup of $G$. We define
		$$\mathcal{L}(G,N, \Lambda) = \langle g^{f(c-i)} \mid g\in \Gamma_i(G^\Q)\cap N \text{ with }1\leq i \leq c\rangle,$$
		where $f: \N\cup\{0\} \to \N$ is defined inductively via $f(0) = \Lambda$ and $f(i+1) = \left(f(i)\Lambda^{c}\right)^c$.
	\end{df}
	In the following results, we will argue that $\mathcal{L}(G,N, \Lambda)$ can be used as the ideal $I$ in Proposition \ref{prop_normal_to_ideal}, proving the result. In fact, $\mathcal{L}(G,N, \Lambda)$ is an LR-group.
	\begin{lemma} \label{lem_kappa_internal}
		Let $g, h \in G$, and let $\kappa(g,h)$ be a basic commutator of length at least $2$ in $\{g,h\}$. If $g\in \mathcal{L}(G,N, \Lambda)$, then $\kappa(g,h) \in \{\tilde{g}^\Lambda\mid \tilde{g}\in \mathcal{L}(G,N, \Lambda)\}$.
	\end{lemma}
	\begin{proof}
		In general, every element $\kappa(g,h)$ with $g\in \mathcal{L}(G,N, \Lambda)$ and $h\in G$ can be written as $\kappa(x_1\cdots x_{k-1}, x_k)$ where every $x_{j}$ for $j \in \{1, \ldots, k-1\}$ is of the form $g_j^{f(c-i)}$ with $g_j \in\Gamma_{i}(G^\Q)\cap N \subset G$ and with $x_k\in G$. 
		We must argue that for any $k \geq 2$ this element lies in $\{\tilde{g}^\Lambda\mid \tilde{g}\in \mathcal{L}(G,N, \Lambda)\}$.
		
		Consider $H = \langle x_1, \ldots , x_k\rangle$, then the element $\kappa(x_1x_2\cdots x_{k-1},x_k)$ lies in $\Gamma_2(H)$. By \cite[Lemma 2.2.3]{robinson} and the fact that $H$ is nilpotent, $\Gamma_2(H)$ is generated by basic commutators of length at least two in $\{x_1, \ldots, x_k\}$. Thus, there is a finite number of basic commutators or their inverses $\kappa^H_j(x_1, x_2, \ldots , x_k)$ with $1\leq j \leq N_3$ for some $N_3 \in \N$ such that
		\begin{equation} \label{eq_free_nilp_kappa}
			\kappa(x_1x_2\cdots x_{k-1},x_k) = \prod_{j=1}^{N_3} \kappa_j^{H}(x_1,x_2, \ldots , x_k).
		\end{equation}
		Also, we may assume that at least one generator in $\{x_1, \ldots , x_{k-1}\}$ appears in each $\kappa_j^{H}$, otherwise $\kappa_j^{H}$ would be trivial. Without loss of generality, suppose $x_1$ appears in $\kappa_j^{H}$, and $x_1 = g_1^{f(c-i)}$ with $g_1 \in \Gamma_i(G^\Q)\cap N$. Since $N$ is normal in $G$, so is $N^{f(c-i)}$. Therefore,
		$$g_2 := \kappa_j^{H}(g_1^{f(c-i)}, x_2, \ldots , x_k) \in N^{f(c-i)} = N^{(f(c-i-1)\Lambda^c)^c}.$$
		
		Now \cite[Chapter 6, Proposition 2]{sega} states that for any nilpotent group $H$ of nilpotency class $c$ and any $s\in \N$, $H^{s^c} \subset \{h^s\mid h\in H\}$, so
		$$g_2 \in \{g^{f(c-i-1)\Lambda^c} \mid g\in N\}.$$
		Write $g_2 = g_3^{f(c-i-1)\Lambda^c}$ with $g_3\in N$. Since $g_1 \in \Gamma_i(G^\Q)$ and $\kappa_j^{H}$ is a basic commutator containing $g_2$, it is clear that $g_2$ and thus also $g_3$ lie in $\Gamma_{i+1}(G^\Q)$. Therefore, $ g_3^{f(c-i-1)} \in \mathcal{L}(G,N, \Lambda)$ and thus $g_2$ is of the form $\tilde{g}^{\Lambda^c}$ with $\tilde{g} \in \mathcal{L}(G,N, \Lambda)$. By Equation \eqref{eq_free_nilp_kappa}, we now conclude that
		$\kappa(x_1x_2\cdots x_{k-1},x_k)$ lies in $\mathcal{L}(G,N, \Lambda)^{\Lambda^c} \subset \{\tilde{g}^\Lambda \mid \tilde{g} \in \mathcal{L}(G,N, \Lambda)\}$, using \cite[Chapter 6, Proposition 2]{sega} again for the last inclusion.
	\end{proof}
	\begin{lemma}
		Let $G\leq (\mathfrak{u}^0(n, \Q), \ast)$ be an $\mathcal{I}$-group. If $N$ is a normal subgroup of $G$, then $\mathcal{L}(G,N, \Lambda)$ is as well.
		Furthermore, if $H \leq \Aut(G^\Q)$ and $G$ and $N$ are $H$-invariant, then $\mathcal{L}(G,N, \Lambda)$ is also $H$-invariant.
	\end{lemma}
	\begin{proof}
		By definition, $\mathcal{L}(G,N, \Lambda)$ is a subgroup. In order to show that it is normal in $G$, it suffices to demonstrate that $[\mathcal{L}(G,N, \Lambda),G]_G \subset \mathcal{L}(G,N, \Lambda)$, which follows directly from Lemma \ref{lem_kappa_internal} with $\kappa(g,h) = [g,h]_G$.
		
		Let $\xi \in H$. Since $\xi(N)= N$ and $\xi(\Gamma_i(G^\Q)) = \Gamma_i(G^\Q)$, we have $\xi(\Gamma_i(G^\Q)\cap N) = \Gamma_i(G^\Q)\cap N$. Thus, for every generator $g^{f(c-i)}$ with $g \in \Gamma_i(G^\Q)\cap N$ of $\mathcal{L}(G,N, \Lambda)$, we have
		$\xi(g^{f(c-i)}) = (\xi(g))^{f(c-i)} \in \mathcal{L}(G,N, \Lambda)$,
		so surely $\xi(\mathcal{L}(G,N, \Lambda)) = \mathcal{L}(G,N, \Lambda)$.
	\end{proof}
	\begin{lemma}
		Let $G\leq (\mathfrak{u}^0(n, \Q), \ast)$ be an $\mathcal{I}$-group with normal subgroup $N$. Then, $\mathcal{L}(G,N, \Lambda)$ is a LR-group. Furthermore, if $G$ is an LR-group with underlying Lie ring structure $L$, then $\mathcal{L}(G,N, \Lambda)$ is an ideal of $L$.
	\end{lemma}
	\begin{proof}
		By Lemma \ref{lem_kappa_internal} and the choice of $\Lambda$, namely such that $r_j, s_j \in (1/\Lambda)\Z$ in Equations \eqref{eq_BCH_inv_1}-\eqref{eq_BCH_inv_2}, we immediately see that
		$$\kappa_j(\mathcal{L}(G,N, \Lambda), G)^{r_j} \in \mathcal{L}(G,N, \Lambda) \text{ and } \tilde{\kappa}_j(\mathcal{L}(G,N, \Lambda), G)^{s_j} \in \mathcal{L}(G,N, \Lambda).$$
		Using this in Equation \eqref{eq_BCH_inv_1} shows that $\mathcal{L}(G,N, \Lambda)$ is additively closed. Equation \eqref{eq_BCH_inv_2} then shows that $[\mathcal{L}(G,N, \Lambda), G]_L \subset \mathcal{L}(G,N, \Lambda)$. In particular, $[\mathcal{L}(G,N, \Lambda), \mathcal{L}(G,N, \Lambda)]_L \subset \mathcal{L}(G,N, \Lambda)$, so $\mathcal{L}(G,N, \Lambda)$ is an LR-group. If $G$ is an LR-group, then this shows that $\mathcal{L}(G,N, \Lambda)$ is an ideal of its underlying Lie ring structure $L$.
	\end{proof}
	\begin{lemma}
		Let $G\leq (\mathfrak{u}^0(n, \Q), \ast)$ be an LR-group with underlying Lie ring structure $L$. If $N$ is a normal subgroup of $G$, then the index of $\mathcal{L}(G,N, \Lambda)$ as a normal subgroup of $G$ and as an ideal of $L$ agree. In particular, there exists a constant $M_2>0$ such that $[L:\mathcal{L}(G,N, \Lambda)] \leq M_2[G:N]$.
	\end{lemma}
	\begin{proof}
		Let $g\in G$. We claim that $g + \mathcal{L}(G,N, \Lambda) = g\ast \mathcal{L}(G,N, \Lambda)$. First, for the inclusion $g+\mathcal{L}(G,N, \Lambda) \subset g\ast \mathcal{L}(G,N, \Lambda)$, take $h\in \mathcal{L}(G,N, \Lambda)$. Now, we find 
		$$g + h = g\ast h \ast \prod_{j=1}^{N_1} (\kappa_j(g,h))^{r_j}.$$
		By Lemma \ref{lem_kappa_internal}, we know that $\prod_{j=1}^{N_1} (\kappa_j(g,h))^{r_j} \in \mathcal{L}(G,N, \Lambda)$, so $g+h = g\ast h'$ with $h' = h\ast \prod_{j=1}^{N_1} (\kappa_j(g,h))^{r_j}\in \mathcal{L}(G,N, \Lambda)$. This ends the first part.
		
		Now, we will proceed to show that for all $g\in G$ and all $h\in \mathcal{L}(G,N, \Lambda)$ there exists $h'\in \mathcal{L}(G,N, \Lambda)$ such that $g\ast h = g+h'$, i.e. the inclusion $g\ast \mathcal{L}(G,N, \Lambda) \subset g+\mathcal{L}(G,N, \Lambda)$. We will do this by induction on $h\in \Gamma_i(G^\Q)$ for $1\leq i \leq c$.
		
		Take first $h\in \Gamma_c(G^\Q)\cap \mathcal{L}(G,N, \Lambda)$. In particular, $h$ is central in $G^\Q$. Therefore, $[g,h]_G = e$ and Equation \eqref{eq_BCH_inv_1} tells us that $g+h = g\ast h$, so the result holds for the base step $i=c$ with $h = h'$.
		
		Suppose the result holds for all $h\in \Gamma_i(G^\Q)\cap \mathcal{L}(G,N, \Lambda)$. Take $h_2 \in \Gamma_{i-1}(G^\Q)\cap \mathcal{L}(G,N, \Lambda)$. Now, by Equation \eqref{eq_BCH_inv_1}
		$$g\ast h_2 = (g+h_2)\ast \left(\prod_{j=1}^\infty (\kappa_j(g,h_2))^{r_j}\right)^{-1}.$$
		We know that $\prod_{j=1}^\infty (\kappa_j(g,h_2))^{r_j}$ lies in $\Gamma_{i+1}(G^\Q)$, so the induction hypothesis can be applied:
		$$(g+h_2)\ast \left(\prod_{j=1}^\infty (\kappa_j(g,h_2))^{r_j}\right)^{-1} = (g+h_2) + h_3$$
		for some $h_3 \in \mathcal{L}(G,N, \Lambda)$. In conclusion, $g\ast h_2 = g + (h_2+h_3)$. By the principal of induction, the general claim now follows.
		
		It is clear that $[G:\mathcal{L}(G,N, \Lambda)] = [L:\mathcal{L}(G,N, \Lambda)]$, since all cosets agree. By construction, we have $N^{f(c-1)} \leq \mathcal{L}(G,N, \Lambda) \leq N$ with $f$ as in Definition \ref{df_LGN}, and thus $[G:\mathcal{L}(G,N, \Lambda)] \leq f(c-1)^{\dim_\Q L^\Q}[G:N]$, where $f(c-1)^{\dim_\Q L^\Q}$ does not depend on the ideal $I$.
	\end{proof}
	This statement finishes the proof of Proposition \ref{prop_normal_to_ideal} by setting $I = \mathcal{L}(G,N, \Lambda)$. 
	\subsection{Proof of Theorem \ref{thm_intro_RFG_RFL}}
		In this subsection, we will prove the two claims made at the start of this section. This includes the proof of Theorem \ref{thm_intro_RFG_RFL}.
		\begin{thm} \label{thm_from_I_to_LR}
			If $G_1$ is an $\mathcal{I}$-group and $H\leq \Aut(G_1)$, then there exists an $H$-invariant LR-group $G_2 \leq G_1$ such that $\RF_{G_1, \mathcal{P}_{H\curvearrowright G_1}} = \RF_{G_2, \mathcal{P}_{H\curvearrowright G_2}}$.
		\end{thm}
		\begin{proof}
			If $H$ is trivial, i.e. $H = \{\Id\}$, then this is essentially Corollary \ref{cor_from_I_to_LR}. In the general case, it suffices to find any $H$-invariant $G_2 \lhd_f G_1$ by Theorem \ref{thm_intro_Q_inv2}. The existence of this $G_2$ is guaranteed by Subsection \ref{sec_normal_to_ideal}. Indeed, if $G$ is an $\mathcal{I}$-group, then we can take, for example, $\mathcal{L}(G,N, \Lambda)$ with $N = G$. This is an $H$-invariant LR-group that is normal in $G$. Its index in $G$ is finite by the observation that $G^{f(c-1)} \leq \mathcal{L}(G,N, \Lambda) \leq G$ with $f$ as in Definition \ref{df_LGN}. 
		\end{proof}
		
		\begin{nota}
			Let $L$ denote a Lie $R$-algebra, and let $H\leq \Aut(L)$. Let $\mathcal{P}_{H\curvearrowright L}$ denote the set of $H$-invariant ideals of $L$.
		\end{nota}
		\begin{proof}[Proof of Theorem \ref{thm_intro_RFG_RFL}]
			Let $G$ be an LR-group with underlying Lie ring structure $L$. Let $H \leq \Aut(G)$. We must prove that  $\RF_{G, \mathcal{P}_{H\curvearrowright G}} = \RF_{L, \mathcal{P}_{H\curvearrowright L}}^{\Gui}$. Note that $\Aut(G) = \Aut(L)$ by Proposition \ref{prop_KQ_maps}, so the notation $\mathcal{P}_{H\curvearrowright L}$ makes sense.
			
			Fix a finite generating set with corresponding word norm $\norm{\cdot}_G$ on $G$ and fix a Guivarc'h length $l_G$ on $L$. By Proposition \ref{prop_Guivarch}, we find a constant $C>0$ such that
			$$\dfrac{1}{C}\norm{g}_G \leq l_G(g) \leq C\norm{g}_G .$$
			
			We start by showing that $\RF_{G, \mathcal{P}_{H\curvearrowright G}} \preceq \RF_{L, \mathcal{P}_{H\curvearrowright L}}^{\Gui}$. Take $0< \norm{g}_G \leq r$. There is an $H$-invariant ideal $I$ of $L$ such that $g\notin I$ and $D_{L, \mathcal{P}_{H\curvearrowright L}}(g) = [L:I]$. Proposition \ref{prop_ideal_to_normal} gives us an $H$-invariant normal subgroup $N\subset I$ such that $[G:N] \leq M_1[L:I]$. Since $g\notin N$, it follows that 
			$$D_{G, \mathcal{P}_{H\curvearrowright G}}(g) \leq [G:N] \leq M_1[L:I] = M_1D_{L, \mathcal{P}_{H\curvearrowright L}}(g).$$
			Since $\norm{g}_G \leq r$, we have $l_G(g) \leq Cr$. Therefore, $D_{L,\mathcal{P}_{H\curvearrowright L}}(g) \leq \RF_{L,\mathcal{P}_{H\curvearrowright L}}^{\Gui}(Cr)$.
			Taking the maximum over all $g\in G$ with $0<\norm{g}_G \leq r$ over the inequality
			$$D_{G,\mathcal{P}_{H\curvearrowright G}}(g) \leq M_1\RF_{L,\mathcal{P}_{H\curvearrowright L}}^{\Gui}(Cr) $$
			shows that $\RF_{G,\mathcal{P}_{H\curvearrowright G}}(r) \leq M_1\RF_{L,\mathcal{P}_{H\curvearrowright L}}^{\Gui}(Cr)$ and therefore $\RF_{G,\mathcal{P}_{H\curvearrowright G}} \preceq \RF_{L, \mathcal{P}_{H\curvearrowright L}}^{\Gui}$.
			
			Completely analogously, one shows the inequality 
			$$\RF_{L, \mathcal{P}_{H\curvearrowright L}}^{\Gui}(r) \leq M_2\RF_{G, \mathcal{P}_{H\curvearrowright G}}(Cr)$$ by using Proposition \ref{prop_normal_to_ideal}, finishing the proof.
		\end{proof}
		By taking $H$ trivial, we see that $\RF_G = \RF_{L}^{\Gui}$ for LR-groups $G$ with underlying Lie ring structure $L$. Combining the previous theorems with Proposition \ref{prop_RF_fin_ext_Hinv}, we obtain the following result:
		\begin{cor} \label{cor_Gtilde_to_L}
			If $\Gamma$ is a virtually nilpotent group, then there exists a Lie ring $L$ and a finitely generated subgroup $H$ of $\Aut(L)$ such that $\RF_{\Gamma} = \RF_{L, \mathcal{P}_{H\curvearrowright L}}$.
		\end{cor}
	\section{Calculation of $\RF_L$} \label{sec_upper}
	In this section, we will focus on Theorem \ref{thm_intro_exact}. In particular, this means that we will define a number $\delta\in\N$ such that $\RF_{L, \mathcal{P}_{H \curvearrowright L}} = \log^\delta$ both for vector norms as Guivarc'h lengths. Here, $L$ is a Lie ring and $H$ is a finitely generated subgroup of $\Aut(L)$. This $\delta$ will be the same as the one given in \cite{dere2025residual}.
	\begin{itemize}
		\item In the first part of Subsection \ref{ssec_delta}, we will define $\delta$ in two different ways. The definition of $\delta$ depends on a choice of a field, and we will therefore write $\delta(L^\F , H)$ in what follows. 
		\item In the second part of Subsection \ref{ssec_delta}, we will show that the value $\delta(L^\F , H)$ is the same for all algebraically closed fields if the characteristic is zero or sufficiently large. 
		\item One of the equivalent definitions of $\delta(L^\F , H)$ uses a formulation in terms of an intersection of ideals $I(L^\F, H)$. Subsection \ref{ssec_delta_Zp} shows that there exists an ideal $I_\delta$ of $L$ such that $I_\delta \otimes_\Z \F = I(L^\F, H)$ for all algebraically closed fields if the characteristic is zero or sufficiently large.
		\item Subsection \ref{ssec_delta_thms} uses these observations to deduce Theorem \ref{thm_intro_exact}.
		\item Finally, the last subsection shows that the $\delta$ from Theorem \ref{thm_intro_exact} only depends on profinite invariants.
	\end{itemize}
	We will write $\Z_{p^k}$ for the ring $\Z/p^k\Z$. In particular, $\Z_p$ will be the field on $p$ elements.
	\subsection{Definition of $\delta$} \label{ssec_delta}
	In this subsection, we will first define the value $\delta$ for a general finite-dimensional Lie algebra $\mathfrak{g}$ over a field $\F$. Then, in Lemma \ref{lem_dimension_alg_closed}, we will restrict our attention to this value for $L^\F$ for algebraically closed fields $\F$.

	\begin{df} \label{df_delta}
		Let $\mathfrak{g}$ denote a finite-dimensional Lie algebra over a field $\F$ and let $H\leq \Aut(\mathfrak{g})$. Define
		\begin{equation*}
			\begin{split}
				\delta_1(\mathfrak{g}, H) & = \min\{\max_{i=1}^k\{\dim_\F \mathfrak{g}/I_i\}\mid I_1, \ldots , I_k \in \mathcal{P}_{H\curvearrowright \mathfrak{g}}, \, \bigcap_{i=1}^k I_i = 0\};\\
				\delta_2(\mathfrak{g}, H) & = \max\{\min\{\dim_\F \mathfrak{g}/I\mid v\notin I\in \mathcal{P}_{H\curvearrowright \mathfrak{g}}\}\mid 0\neq v \in \mathfrak{g}\}.
			\end{split}
		\end{equation*}
	\end{df}
	The definition $\delta_1(\mathfrak{g}, H)$ is the one that can be found in \cite{dere2025residual}, and we show that it is equal to the value $\delta_2(\mathfrak{g}, H)$. Before doing so, let us remark that these values can be determined on concrete examples. This is done for example in the papers \cite{math_virt_ab, dere2025twostep} for virtually abelian groups and two-step nilpotent groups. In the virtually abelian setting, we see that $\delta_1(\C^n,H)$ equals the largest dimension of an absolutely irreducible subspace with respect to the action of $H$.
	
	Let us now first rephrase the meaning of $\delta_2(\mathfrak{g}, H)$.
	\begin{lemma} \label{lem_delta2_intersections}
		Let $\mathfrak{g}$ denote a finite dimensional Lie algebra over a field $\F$, and $H\leq \Aut(\mathfrak{g})$. Then, $\delta_2(\mathfrak{g}, H)$ is the number such that
		$$\bigcap\{I \in \mathcal{P}_{H\curvearrowright \mathfrak{g}}\mid \dim_\F \mathfrak{g}/I < \delta_2(\mathfrak{g}, \F)\} \neq 0$$
		and $$\bigcap\{I \in \mathcal{P}_{H\curvearrowright \mathfrak{g}}\mid \dim_\F \mathfrak{g}/I \leq \delta_2(\mathfrak{g}, \F)\} = 0.$$
	\end{lemma}
	\begin{proof}
		Take a non-zero vector $v\in\mathfrak{g}$ realizing the maximum in the definition of $\delta_2(\mathfrak{g}, H)$. Now,
		$$\delta_2(\mathfrak{g}, H) = \min\{\dim_\F \mathfrak{g}/I\mid v\notin I\in \mathcal{P}_{H\curvearrowright \mathfrak{g}}\},$$
		so $v \in J$ for all $J\in \mathcal{P}_{H\curvearrowright \mathfrak{g}}$ with $\dim_\F \mathfrak{g}/J < \delta_2(\mathfrak{g}, H)$. Hence,
		$$0 \neq v \in \bigcap\{I \in \mathcal{P}_{H\curvearrowright \mathfrak{g}}\mid \dim_\F \mathfrak{g}/I < \delta_2(\mathfrak{g}, H)\}.$$
		Suppose by contradiction that $\{I \in \mathcal{P}_{H\curvearrowright \mathfrak{g}}\mid \dim_\F \mathfrak{g}/I \leq \delta_2(\mathfrak{g}, H)\}$ has a non-trivial intersection, say $0\neq w\in \mathfrak{g}$ lies in this intersection, then
		$$\min\{\dim_\F \mathfrak{g}/I\mid w\notin I\in \mathcal{P}_{H\curvearrowright \mathfrak{g}}\} > \delta_2(\mathfrak{g}, H).$$
		However, by definition, $\min\{\dim_\F \mathfrak{g}/I\mid w\notin I\in \mathcal{P}_{H\curvearrowright \mathfrak{g}}\} \leq \delta_2(\mathfrak{g}, H)$, contradicting the inequality above. Hence, such a non-trivial $w\in \mathfrak{g}$ does not exist, i.e.
		$$\bigcap\{I \in \mathcal{P}_{H\curvearrowright \mathfrak{g}}\mid \dim_\F \mathfrak{g}/I \leq \delta_2(\mathfrak{g}, H)\} = 0.$$
	\end{proof}
	\begin{lemma}
		Let $\mathfrak{g}$ denote a finite dimensional Lie algebra over a field $\F$, and $H \leq \Aut(\mathfrak{g})$. Then, $\delta_1(\mathfrak{g}, H) = \delta_2(\mathfrak{g}, H)$.
	\end{lemma}
	\begin{proof}
		Since $\mathfrak{g}$ is finite dimensional and
		$$\bigcap\{I \in \mathcal{P}_{H\curvearrowright \mathfrak{g}}\mid \dim_\F \mathfrak{g}/I \leq \delta_2(\mathfrak{g}, H)\} = 0,$$
		we can surely take a finite subset
		$$\{I_i\mid 1\leq i \leq k\} \subset \{I \in \mathcal{P}_{H\curvearrowright \mathfrak{g}}\mid \dim_\F \mathfrak{g}/I \leq \delta_2(\mathfrak{g}, H)\}$$
		with trivial intersection. By definition, this implies that $\delta_1(\mathfrak{g}, H) \leq \delta_2(\mathfrak{g}, H)$.
		
		On the other hand, one cannot take a finite subset of $\{I \in \mathcal{P}_{H\curvearrowright \mathfrak{g}}\mid \dim_\F \mathfrak{g}/I \leq \delta_2(\mathfrak{g}, H)-1\}$ with trivial intersection, since
		$$\bigcap\{I\in \mathcal{P}_{H\curvearrowright \mathfrak{g}}\mid \dim_\F \mathfrak{g}/I \leq \delta_2(\mathfrak{g}, H)-1\} \neq 0. $$
		Therefore, surely, $\delta_1(\mathfrak{g}, H) > \delta_2(\mathfrak{g}, H)-1$, or equivalently $\delta_1(\mathfrak{g}, H) \geq \delta_2(\mathfrak{g}, H)$.
	\end{proof}
	\begin{nota}
		If $L$ is a Lie ring, then we have its completion $L^{\F} = L \otimes_\Z \F$. Note that this also makes sense when $\F$ is a field of prime characteristic. If $H \leq \Aut(L)$ is given, we will reinterpret it as $H \leq \Aut(L^\F)$ via $\{\xi \otimes_\Z \Id_\F \mid \xi \in H\} \leq \Aut(L^\F)$. 
	\end{nota}
	\begin{nota}
		Let $\mathcal{L}_{\rng}$ be the first order language of rings, i.e. the first order language with signature $\{0,1,+,-, \cdot\}$.
	\end{nota}
	\begin{lemma} \label{lem_dimension_alg_closed}
		Let $L$ denote a Lie ring and $H$ a finitely generated subgroup of $\Aut(L)$. For every $k\in \N$, there exists a number $M_k \in \N$ such that for any algebraically closed field $\F$ of characteristic zero or characteristic $p> M_k$ we have
		\begin{equation*}
			\dim_\F\bigcap\{I \in \mathcal{P}_{H\curvearrowright L^\F}\mid \dim_\F L^\F/I \leq k\} = \dim_\C\bigcap\{I \in \mathcal{P}_{H\curvearrowright L^\C}\mid \dim_\C L^\C/I \leq k\}.
		\end{equation*}
		In particular, there exists a number $M_\delta \in \N$ such that for any algebraically closed field $\F$ of characteristic zero or characteristic $p> M_\delta$ we have $\delta_2(L^\F, H) = \delta_2(L^\C, H)$.
	\end{lemma}
	\begin{proof}
	Set
	$$d_k = \dim_\C\bigcap\{I \in \mathcal{P}_{H\curvearrowright L^\C}\mid \dim_\C L^\C/I \leq k\}.$$
	For the first statement, it suffices to show that the dimension above can be expressed as a sentence in the first order language of rings, $\mathcal{L}_{\rng}$. 
	Indeed, by the Lefschetz Principle given in \cite[Theorem 3.5.5]{logic}, the fact that  $$\dim_\F\bigcap\{I \in \mathcal{P}_{H\curvearrowright L^\F}\mid \dim_\F L^\F/I \leq k\} = d_k$$
	holds for $\F = \C$ implies that it holds for any algebraically closed field of characteristic zero or $p$ with $p > M_k$ for some number $M_k$. This implies the last statement as well, by taking $M_\delta = \max\{M_k \mid 0\leq k \leq n\}$, because then the first statement holds for all $k \in \N$ and $p > M_\delta$. In particular $\delta_2(L^\F, \F) = \delta_2(L^\C, \C)$ follows from Lemma \ref{lem_delta2_intersections}.
	
In order to show that the dimension can be expressed as a sentence in first order language, fix a field $\F$ and a $\Z$-basis of the Lie ring $L$. Vectors in $L^\F$ are then identified with their coordinates with respect to this fixed basis. Let $\{\xi_l \mid 1\leq l \leq s\}$ be a set of generators of $H\leq \Aut(L)$, supplemented with their inverses. With respect to the given basis, they are represented by integral matrices. Note that integers can be expressed in the language of rings.

If a vector $v\in L^\F$ is given, we have the equivalence
	\begin{equation} \label{eq_in_intersection}
		v \in \bigcap\{I \in \mathcal{P}_{H\curvearrowright L^\F}\mid \dim_\F L^\F/I = k\} \Leftrightarrow \left(\forall I \in \mathcal{P}_{H\curvearrowright L^\F}: \dim_\F L^\F/I = k\Rightarrow v \in I\right)		
	\end{equation}
Now, an ideal $I$ is represented by a set of vectors $\{w_1, \ldots , w_j\}$ that span the ideal. This way, the phrase $v\in I$ can be rephrased to 
		$$\exists \lambda_1, \ldots , \lambda_j \in \F: v = \lambda_1w_1 + \ldots + \lambda_jw_j.$$
		Furthermore, the part $\forall I \in \mathcal{P}_{H\curvearrowright L^\F}: \dim_\F L^\F/I = k$ can be encoded as
		\begin{multline*}
			\forall \{w_1, \ldots ,w_n\} \in L^\F: \det\{w_1, \ldots ,w_n\} \neq 0 \\
			\text{and } \forall 1\leq i\leq n-k:\forall 1\leq j\leq n: [w_i,w_j]_L \in \vct_\F\{w_1, \ldots , w_{n-k}\} \\
			\text{and } \forall 1\leq l \leq s: \forall 1\leq i \leq n-k: \xi_l(w_i) \in \vct_\F\{w_1, \ldots , w_{n-k}\}. 
		\end{multline*}
		Indeed, the ideal $I$ with $\dim_\F L^\F/I = k$ is generated by the first $n-k$ vectors of a basis of $L^\F$. The second part of the statement checks whether these base vectors span an ideal. Since $L$ is a Lie ring, the bracket $[\cdot, \cdot]_L$ has integral structure constants, therefore $[w_i,w_j]$ is an integral polynomial in the coordinates of $w_i$ and $w_j$. The third part of the statement checks whether $I$ is $H$-invariant.
		
		Combined, we have rephrased Equation \eqref{eq_in_intersection} using coordinates as a sentence in $\mathcal{L}_{\rng}$.  Hence, we can also rephrase 
		\begin{multline*}
			v \in \bigcap\{I \in \mathcal{P}_{H\curvearrowright L^\F}\mid \dim_\F L^\F/I \leq k\} \Leftrightarrow v \in \bigcap\{I \in \mathcal{P}_{H\curvearrowright L^\F}\mid \dim_\F L^\F/I = k\}\\ \text{and } \ldots \text{ and }v \in \bigcap\{I \in \mathcal{P}_{H\curvearrowright L^\F}\mid \dim_\F L^\F/I = 0\}
		\end{multline*}
	by splitting up the dimension.	Fixing $\delta\in \N$, the expression
		\begin{equation*} \label{eq_dimension_field}
			\dim_\F\left(\bigcap\{I \in \mathcal{P}_{H\curvearrowright L^\F}\mid \dim_\F L^\F/I \leq k\}\right) = \delta 
		\end{equation*}
		can now be rephrased  to the conjunction of the two expressions 
		\begin{multline*}
			\exists \{v_1, \ldots ,v_n\} \in L^\F: \det\{v_1, \ldots ,v_n\} \neq 0 
			\text{ and }\forall 1\leq j \leq \delta: v_j \in \bigcap\{I \in \mathcal{P}_{H\curvearrowright L^\F}\mid \dim_\F L^\F/I \leq k\}
		\end{multline*}
which states that $\dim \geq \delta$ and		\begin{multline*}
			\forall \{w_1, \ldots ,w_n\} \in L^\F: \det\{w_1, \ldots ,w_n\} \neq 0 \\
			\text{and }\forall 1\leq j \leq \delta: w_j \in \bigcap\{I \in \mathcal{P}_{H\curvearrowright L^\F}\mid \dim_\F L^\F/I \leq k\} \\ \Rightarrow w_{\delta+1} \notin \bigcap\{I \in \mathcal{P}_{H\curvearrowright L^\F}\mid \dim_\F L^\F/I \leq k\},
		\end{multline*}
		which states that $\dim \leq \delta$. We conclude that the dimension is expressed as a sentence in $\mathcal{L}_{\rng}$ and thus the lemma holds.
	\end{proof}
	Reduction from an algebraically closed field of characteristic $p$ to $\Z_p$ is possible for infinitely many prime, as shown in \cite[Proposition 6.7 \& Corollary 6.8]{dere2025residual}, by making use of Chebotarev's Density Theorem. This insight leads to the upper bound of Theorem \ref{thm_intro_exact}.
	\begin{lemma}\label{lem_C_Zp_corr}
		Let $\norm{\cdot}_L$ denote a norm on $L$. There exists a constant $C>0$ such that for every non-trivial $v\in L$ with $\norm{v}_L \leq r$, there exists a prime $p\leq C\log(r)+C$ such that $v \notin pL$ and $\delta_1(L^{\Z_p}, H) \leq \delta_1(L^\C, H)$. In particular, there are infinitely many primes $p$ such that $\delta_1(L^{\Z_p}, H) \leq \delta_1(L^\C, H)$.
	\end{lemma}
	\subsection{The ideal $I_\delta$} \label{ssec_delta_Zp}
	By Lemma \ref{lem_dimension_alg_closed}, we know that
	$$\dim_{\F_p}\bigcap\{I \in \mathcal{P}_{H\curvearrowright L^{\F_p}}\mid \dim_{\F_p} L^{\F_p}/I < \delta_1(L^\C, H)\} = \dim_\C\bigcap\{I \in \mathcal{P}_{H\curvearrowright L^{\C}}\mid \dim_\C L^\C/I < \delta_1(L^\C, H)\},$$
	where $\F_p$ is an algebraically closed field of sufficiently large prime characteristic. However, although the dimensions are equal, it is still unclear whether there is a relation between these two ideals themselves. In this subsection, we show that there exists an ideal $I_\delta$ of $L$ such that both can be obtained as $I_\delta \otimes_\Z \F$, where $\F = \F_p$ and $\F = \C$ respectively. This will be done by first passing to a number field, and then we will apply the theory of Galois Descent.
	\begin{lemma} \label{lem_dim_number_field}
		Let $L$ denote a Lie ring, $H \leq \Aut(L)$ finitely generated, and $k\in\N$. There exists a number field $\F$ such that
		\begin{equation*}
			\dim_\F\bigcap\{I \in \mathcal{P}_{H \curvearrowright L^\F}\mid \dim_\F L^\F/I < k\} = \dim_\C\bigcap\{I \in\mathcal{P}_{H \curvearrowright L^\C}\mid \dim_\C L^\C/I < k\}.
		\end{equation*}
		In particular, there is a number field such that $\delta_2(L^\C, H) = \delta_2(L^\F, H)$.
	\end{lemma}
	\begin{proof}
			To conclude the final statement from the first, take for every $1 \leq l \leq \dim_\Z L+1$ the number fields $\F_l$ for which the given dimensions are equal. Take any number field $\F$ that contains all the fields $\F_l$ above, then the result follows from the alternative description of $\delta_2$ in Lemma \ref{lem_delta2_intersections}.
		
		So it suffices to show the first statement about the dimensions.
		By Lemma \ref{lem_dimension_alg_closed}, we know that
		\begin{equation*}
			\dim_\A\bigcap\{I \in \mathcal{P}_{H \curvearrowright L^\A}\mid \dim_\A L^\A/I < k\} = \dim_\C\bigcap\{I \in \mathcal{P}_{H \curvearrowright L^\C}\mid \dim_\C L^\C/I < k\},
		\end{equation*}
		where $\A$ is the algebraic closure of $\Q$. We will now argue that there exists a number field $\F$ such that
		$$\bigcap\{I \in \mathcal{P}_{H \curvearrowright L^\F}\mid \dim_\F L^\F/I < k\} \otimes_\F \A = \bigcap\{I\in \mathcal{P}_{H \curvearrowright L^\A}\mid \dim_\A L^\A/I < k\},$$
		from which the lemma's statement clearly follows.
		
		Note first that for any number field $\F$, we have
		$$\bigcap\{I \in \mathcal{P}_{H \curvearrowright L^\A}\mid \dim_\A L^\A/I < k\} \subset \bigcap\{I \in \mathcal{P}_{H \curvearrowright L^\F}\mid \dim_\F L^\F/I < k\} \otimes_\F \A .$$
		Since $L$ is finite dimensional, we however know that
		$$\bigcap\{I \in \mathcal{P}_{H \curvearrowright L^\A}\mid \dim_\A L^\A/I < k\} = I_1 \cap \ldots \cap I_l $$
		for a finite subset $\{I_1, \ldots ,I_l\}$ of $\{I \in\mathcal{P}_{H \curvearrowright L^\A}\mid \dim_\A L^\A/I < k\}$. Fix a $\Z$-basis of $L$. Every ideal $I_i$ equals $\vct_\A\{w_1^{(i)}, \ldots, w_{j_i}^{(i)}\}$ for some vectors in $L^\A$. Identifying them with their coordinates, we see that the finite set of vectors
		$$\{w_1^{(i)}, \ldots, w_{j_i}^{(i)}\mid 1\leq i \leq l\}$$
		must be included in $L^\F$ for some number field $\F$. Set $I_i^\F = \vct_\F\{w_1^{(i)}, \ldots, w_{j_i}^{(i)}\}$. Then, $I_i^\F \otimes_\F \A = I_i$, and therefore,
		$$\bigcap\{I \in \mathcal{P}_{H \curvearrowright L^\F}\mid \dim_\F L^\F/I < k\} \otimes_\F \A \subset I_1 \cap \ldots \cap I_l,$$
		ending the first part.
	\end{proof}
	\begin{lemma} \label{lem_Idelta}
		There exists an ideal $I_{\delta}$ in $L$ and a number $M_{I, \delta} \in \N$ such that for any algebraically closed field $\F$ of characteristic zero or characteristic $p> M_{I, \delta}$ we have
		$$I_\delta\otimes_\Z \F = \bigcap\{I \in \mathcal{P}_{H \curvearrowright L^\F}\mid \dim_\F L^\F/I < \delta_2(L^\C, H)\} .$$
	\end{lemma}
	\begin{proof}
		Set $\delta = \delta_2(L^\C, H)$. By Lemma \ref{lem_dim_number_field}, there exists a number field $\F_\delta \subset \C$ such that
		\begin{equation} \label{eq_equal_dim}
			\dim_{\F_\delta}\bigcap\{I \in \mathcal{P}_{H\curvearrowright L^{\F_\delta}}\mid \dim_{\F_\delta} L^{\F_\delta}/I < \delta\} = \dim_\C\bigcap\{I \in \mathcal{P}_{H \curvearrowright L^\C}\mid \dim_\C L^\C/I < \delta\}.
		\end{equation}
		By enlarging the field if needed, we may suppose $\F_\delta$ is Galois over $\Q$ with Galois group $G := \Gal(\F_\delta/\Q)$.
		
		Fix a $\Z$-basis $\{v_i\mid 1\leq i \leq n\}$ of $L$. Define for every $\sigma \in G$ the map
		$$\sigma: L^{\F_\delta} \to L^{\F_\delta}: v = \sum_{i=1}^n \lambda_i v_i \mapsto \sigma(v):= \sum_{i=1}^n \sigma(\lambda_i) v_i,$$
		i.e. applying $\sigma$ coordinate-wise to a vector. These maps satisfy the following properties.
			\begin{enumerate}[(i)]
			\item It defines an action of $G$ on $L^{\F_\delta}$, namely $\Id_G(v) = v$ and $\sigma_1(\sigma_2(v)) = (\sigma_1\circ\sigma_2)(v)$ for all $v\in L^{\F_\delta}$ and $\sigma_1, \sigma_2 \in G$.
			\item The maps are $\sigma$-linear, so $\sigma(v + w) = \sigma(v) + \sigma(w)$ and $\sigma(\lambda v) = \sigma(\lambda)\sigma(v)$ for all $v,w\in L^{\F_\delta}$ and $\lambda \in \F_\delta$.
			\item The maps preserve the Lie bracket, i.e.~$\sigma([v,w]_L) = [\sigma(v), \sigma(w)]_L$, since $[\cdot , \cdot]_L$ has integral structure constants.
			\item If $\xi \in H$, then $\sigma(\xi(v)) = \xi(\sigma(v))$, since $\xi\in \Aut(L)$.
		\end{enumerate}
		According to the terminology of \cite[Definition 3.2.2]{Winter}, the first two points say that the construction defines a $G$-product on the vector space underlying $L^{\F_\delta}$. The third point says that the action maps ideals to ideals of the same dimension. The fourth point guarantees that $H$-invariance is preserved. Hence, for all $\sigma \in G$
		\begin{equation*}
			\sigma\left(\bigcap\{I \in \mathcal{P}_{H\curvearrowright L^{\F_\delta}}\mid \dim_{\F_\delta} L^{\F_\delta}/I < \delta\}\right)\\ = \bigcap\{I \in \mathcal{P}_{H\curvearrowright L^{\F_\delta}}\mid \dim_{\F_\delta} L^{\F_\delta}/I < \delta\}.
		\end{equation*}
		
		By the theory of Galois Descent (see \cite[Theorem 3.2.5]{Winter}), this implies that their exists a $\Q$-vector space $V \subset L^\Q$ such that
		$$V\otimes_\Q \F_\delta = \bigcap\{I \in \mathcal{P}_{H\curvearrowright L^{\F_\delta}}\mid \dim_{\F_\delta} L^{\F_\delta}/I < \delta\}.$$
		Note that $V$ must be an ideal. We claim that $I_\delta = V\cap L$ is the ideal of $L$ for which the lemma's statement holds.
		
		Take a $\Z$-basis $\{w_1, \ldots ,w_l\}$ of $I_\delta$ ($l\in \N$). By Equation \eqref{eq_equal_dim}, we know that
		\begin{equation*}
			\bigcap\{I \in \mathcal{P}_{H\curvearrowright L^{\F_\delta}}\mid \dim_{\F_\delta} L^{\F_\delta}/I < \delta\}\otimes_{\F_\delta} \C = \bigcap\{I \in \mathcal{P}_{H\curvearrowright L^{\C}}\mid \dim_\C L^\C/I < \delta\},
		\end{equation*}
		so we know
		$$\forall \lambda_1, \ldots , \lambda_l \in \C: \sum_{i=1}^l \lambda_iw_i \in  \bigcap\{I \in \mathcal{P}_{H\curvearrowright L^{\C}}\mid \dim_\C L^\C/I < \delta\}.$$
		
		Just as in the proof of Lemma \ref{lem_dimension_alg_closed}, this can be expressed as a sentence in $\mathcal{L}_{\rng}$. Again, by Lefschetz' Principle, see \cite[Theorem 3.5.5]{logic}, the fact that this holds over $\C$ implies that it holds over any algebraically closed field $\F$ of characteristic zero or $p$, if $p$ is sufficiently large, i.e.
		$$I_\delta \otimes_\Z \F \subset \bigcap\{I \in \mathcal{P}_{H\curvearrowright L^{\F}}\mid \dim_\F L^\F/I < \delta\}.$$
		We end the proof by noting that over  characteristic zero or a sufficiently large prime, Lemma \ref{lem_dimension_alg_closed} guarantees that the dimensions on both sides are equal.
	\end{proof}
\subsection{Proof of Theorem \ref{thm_intro_exact}} \label{ssec_delta_thms}
	In this subsection, we will use the results of the previous subsections to give proof of Theorem \ref{thm_intro_exact}. It will build upon Theorem \ref{thm_RF_Pinfty} below and Corollary \ref{cor_Gtilde_to_L}. The upper bound statement was given in \cite{dere2025residual}, but we will deduce it again, as it only requires Lemma \ref{lem_C_Zp_corr} originating from that paper to do so.
\begin{lemma} \label{lem_Q_is_p_powered}
	Let $L$ denote a Lie ring. If $0\neq v \in L$ and $\psi: L \to Q$ is a homomorphism with $H$-invariant kernel such that $\psi(v) \neq 0$ and $D_{L, \mathcal{P}_{H\curvearrowright L}}(v) = |Q|$, then $|Q| = p^l$ for some prime $p$ and power $l\in\N$.
\end{lemma}
\begin{proof}
	The set $Q$ is a finite $\Z$-module and, hence, there are distinct primes $\{p_i \mid 1\leq i \leq d\}$ such that $Q$ can be decomposed as a direct sum of $\Z$-modules
	$$Q = M_1 \oplus M_2 \oplus \ldots \oplus M_d,$$
	where $|M_i| = p_i^{l_i}$. In fact, this is also a direct sum of Lie $\Z$-algebras. Indeed, let $v_i \in M_i$ and $v_j \in M_j$. Then,
	$$p_i^{l_i}[v_i,v_j]_L = [p_i^{l_i}v_i,v_j]_L = 0 \text{ and }p_j^{l_j}[v_i,v_j]_L = [v_i,p_j^{l_j}v_j]_L = 0,$$
	so $[v_i,v_j]_L = \gcd(p_i^{l_i},p_j^{l_j})[v_i,v_j]_L = 0$ by Bezout’s identity. Furthermore, $\psi^{-1}(M_i)$ is $H$-invariant. Indeed, $0\neq w\in \psi^{-1}(M_i)$ if and only if some $p_i^{l_i}$ exists such that $p_i^{l_i}w \in \ker\varphi$. Now, let $\xi\in H$ and $v_i\in \psi^{-1}(M_i)$, then $p_i^{l_i}\xi(v_i) = \xi(p_i^{l_i}v_i) \in \xi(\ker\varphi) = \ker\varphi$, so $\xi(\psi^{-1}(M_i)) = \psi^{-1}(M_i)$.
	
	Now, if $\psi(v) \neq 0$, then it is non-zero in one of the summands, say $M_1$. Thus, $(\pi\circ\psi)(v) \neq 0$ with $\pi : Q \to M_1$ the projection onto that component. The kernel of this map is $H$-invariant by the previous observation. This implies that $D_{L, \mathcal{P}_{H\curvearrowright L}}(v) \leq |M_1|$. Since $|Q| = D_{L,\mathcal{P}_{H\curvearrowright L}}(v)$, we conclude that $Q = M_1$. 
\end{proof}
\begin{lemma} \label{lem_gaps_in_estimates}
	Let $f,g: \N\to \R_{\geq 1}$ be increasing functions. Assume that there exists an $s\in \N$ with $s>1$ such that $f(s^l) \leq g(s^{l})$ for all $l\in \N$ sufficiently large, then $f(r) \preceq g(r)$.
\end{lemma}
\begin{proof}
	We need to show that there exists a constant $C>0$ such that $f(r) \leq Cg(Cr)$. Suppose the inequality $f(s^l) \leq g(s^{l})$ holds for all $l\geq l_0$. Take $r\geq s^{l_0}$ arbitrary. Now, take $l \geq l_0$ such that $s^l \leq r < s^{l+1}$. Now, since $f$ is increasing, we have
	$$f(r) \leq f(s^{l+1}) \leq g(s^{l+1}) \leq g(sr).$$
	Now, take $C = \max\{s,f(s^{l_0})\}$, then $f(r) \leq Cg(Cr)$ for all $r\in\N$.
\end{proof}
Where we used the Lefschetz Principle in the previous subsection, we will need another result in what follows, namely some form of the Ax-Kochen Principle (see \cite[Theorem 2.14]{MR3013956}).
\begin{thm}[Ax-Kochen Principle] \label{thm_ax_kochen}
	Let $\sigma$ be a sentence in $\mathcal{L}_{\rng}\cup\{t\}$ and $k\in \N$. The statement $\sigma$ is true in $\Z_{p^k}$ if and only if the statement is true in $\Z_p[T]/(T^k)$ for all prime numbers $p$ sufficiently large. Here, the constant $t$ is identified with $p$ and $T$, respectively.
\end{thm}
\begin{ex}
	Fix $k\in \N$. The expression $\forall x: t^kx = 0$ is a sentence in $\mathcal{L}_{\rng}\cup \{t\}$, as it only uses symbols in $\mathcal{L}_{\rng}$ and the new symbol $t$. This statement is true in $\Z_{p^k}$ because $t$ becomes $p$ in this setting. The statement is also true in $\Z_p[T]/(T^k)$ with $t$ becoming $T$, as $$\forall x \in \Z_p[T]/(T^k): T^kx = 0.$$
\end{ex}
We will use this principle to derive the following crucial insight:
\begin{prop} \label{prop_Zpk_to_Zp}
Let $L$ be a nilpotent Lie ring, $v\in L$ and $H\leq \Aut(L)$ be finitely generated. For every $m\in \N$, there exists a bound $M\in \N$ such that for all primes $p \geq M$ and all $I \lhd L^{\Z_{p^k}}$ an $H$-invariant ideal such that $v\notin I$ and $[L^{\Z_{p^k}}: I] \leq p^m$, there exists an $H$-invariant ideal $J \lhd L^{\Z_p}$ such that $v\notin L^{\Z_p}$ and $|L^{\Z_p}/J|\leq |L^{\Z_{p^k}}/I|$.
\end{prop}
\begin{proof}
	First, let us describe the setup on which we will apply the Ax-Kochen Principle. This setup is essentially the same as the setup we used before when we dealt with Lefschetz' Principle.
	
	Take a $\Z$-basis of the $n$-dimensional $L$, and describe all vectors in $L$, $L^{\Z_{p^k}}$ and $L\otimes_\Z (\Z_p[T]/(T^k))$ using coordinate vectors with entries in $\Z$, $\Z_{p^k}$ and $\Z_p[T]/(T^k)$ respectively. Since $v\in L$, it is described by coordinates in $\Z$, and it can therefore easily be represented in the language of rings. In what follows, the vectors $w_i$ will represent tuples of length $n$, corresponding to vectors in $L^{\Z_{p^k}}$ or $L\otimes_\Z (\Z_p[T]/(T^k))$. Since $H$ is finitely generated, we can take generators $\xi_1$ to $\xi_s$. They are given by integral matrices. The Lie bracket $[\cdot, \cdot]_L$ is defined by its integral structure constants.
	
	Suppose $(e_1, \ldots , e_n) \in (\N\cup\{0\})^n$ is fixed. Consider the following sentence in the language $\mathcal{L}_{\rng}\cup\{t\}$:
	\begin{multline} \label{eq_ideal_koch}
		\exists \{w_1, \ldots , w_n\}: \exists \lambda: \det\{w_1, \ldots , w_n\}\cdot \lambda = 1 \text{ and } v \notin \vct\{t^{e_1}w_1, \ldots , t^{e_n}w_n\}\\
		\text{and } \forall 1\leq i,j\leq n: [t^{e_i}w_i,w_j]_L \in \vct\{t^{e_1}w_1, \ldots , t^{e_n}w_{n}\} \\
		\text{and } \forall 1\leq l \leq s: \forall 1 \leq i \leq n: \xi_l(t^{e_i}w_i) \in \vct\{t^{e_1}w_1, \ldots , t^{e_n}w_{n}\}.
	\end{multline}
	Recall that the statement $v \notin \vct\{t^{e_1}w_1, \ldots , t^{e_n}w_n\}$ is shorthand notation for
	$$\forall \mu_1, \ldots , \mu_n: v \neq \mu_1t^{e_1}w_1 + \ldots + \mu_n t^{e_n}w_n.$$
	By Theorem \ref{thm_ax_kochen}, we know that this statement holds in $\Z_{p^k}$ if and only if it holds in $\Z_p[T]/(T^k)$, provided that the prime number $p$ is sufficiently large. Take $M\in\N$ such that this result applies to all sentences with $\sum_{j=1}^n e_j \leq m$.
	
	Now, suppose an $H$-invariant ideal $v\notin I\lhd L^{\Z_{p^k}}$ is given for $p\geq M$ and $[L^{\Z_{p^k}}:I] \leq p^m$. This ideal corresponds to an ideal $\tilde{I}$ in $L$ itself. As a $\Z$-submodule of $L$, we know there are vectors $\{w_1, \ldots , w_n\}$ and numbers $\{e_1, \ldots , e_n\}$ such that $\{w_1, \ldots , w_n\}$ is a basis of $L$ and
	$$\tilde{I} = \vct_\Z\{p^{e_1}w_1, \ldots, p^{e_n}w_n\}.$$
	Note that $\sum_{j=1}^n e_i \leq k$, since this sum gives the index of $\tilde{I}$ in $L$ and thus of $I$ in $L^{\Z_{p^k}}$.
	
	Using these vectors (interpreted modulo $p^k$) and this choice of integers $(e_1, \ldots , e_n)$, we see that the statement in Equation \eqref{eq_ideal_koch} holds with $t = p$. Indeed, the first line follows by the fact that $\{w_1, \ldots , w_n\}$ was a $\Z$-basis and the statement that $v\notin I$. The second line says that $I$ defines an ideal. The third line says that this ideal is $H$-invariant. Therefore, since $p\geq M$, the Ax-Kochen Principle says that this statement also holds in $\Z_p[T]/(T^k)$ with $t= T$. We will now continue in this setting.
	
	By Equation \eqref{eq_ideal_koch}, we obtain a set $\{w_1, \ldots, w_n\} \subset \Z_p[T]/(T^k)$ and 
	$$\tilde{J} = \vct_{\Z_p[T]/(T^k)}\{T^{e_1}w_1, \ldots , T^{e_n}w_n\}$$
	such that $v\notin \tilde{J}$. Since $\det\{w_1, \ldots , w_n\}\cdot \lambda = 1$, we see that $\{w_1, \ldots , w_n\}$ spans $L\otimes_\Z \Z_p[T]/(T^k)$. Hence, the second line of the equation says that $\tilde{J}$ is an ideal of $L\otimes_\Z \Z_p[T]/(T^k)$. The third line guarantees that $\tilde{J}$ is $H$-invariant. Furthermore, from the fact that $\{w_1, \ldots , w_n\}$ spans $L\otimes_\Z \Z_p[T]/(T^k)$ and from the definition of $\tilde{J}$, it follows easily that 
	$$\left|\dfrac{L\otimes_\Z \Z_p[T]/(T^k)}{\tilde{J}} \right| = p^{\sum_{j=1}^n e_j}.$$
	
	Consider $J = \tilde{J}\cap L^{\Z_p}$. Since $\tilde{J}$ is an $H$-invariant ideal over the ring $\Z_p[T]/(T^k)$, $J$ is surely an $H$-invariant ideal of $L^{\Z_p}$ over the ring $\Z_p$. Since $v\notin \tilde{J}$, we also see that $v\notin J$. It suffices to argue that $|L^{\Z_p}/J| \leq p^{\sum_{j=1}^n e_j}$. For this, note that if $v_1, v_2 \in L^{\Z_p}$, then $v_1 + \tilde{J} \neq v_2 + \tilde{J}$ if and only if $v_1+J \neq v_2+ J$ and hence
	$$\left|\dfrac{L\otimes_\Z \Z_p[T]/(T^k)}{\tilde{J}} \right| \geq \left| \dfrac{L^{\Z_p}}{J}\right| .$$
\end{proof}
\begin{thm} \label{thm_RF_Pinfty}
	We have $\RF_{L, \mathcal{P}_{H\curvearrowright L}}^{\ast} = \log^{\delta_1(L^\C, H)}$ with $\ast\in\{\nom, \Gui\}$.
\end{thm}
\begin{proof}
	Let us first argue that $\RF_{L, \mathcal{P}_{H\curvearrowright L}}^{\nom} \preceq \log^{\delta_1(L^\C, H)}$: Take a non-trivial vector $v\in L$ with $\norm{v}_L \leq r$. By Lemma \ref{lem_C_Zp_corr}, there exists a prime $p \leq C\log(r) + C$ such that $v\notin pL$ and $\delta_1(L^\C , H) \geq \delta_1(L^{\Z_p}, H)$. By definition of $\delta_1$, we know that there exists an ideal $I \in \mathcal{P}_{H\curvearrowright L^{\Z_p}}$ that does not contain the non-trivial vector $v+pL$ and such that $\dim_{\Z_p} L^{\Z_p}/I \leq \delta_1(L^{\Z_p}, H)$. We conclude that
	\begin{equation*}
		D_{L, \mathcal{P}_{H\curvearrowright L}}(v) \leq |L^{\Z_p}/I| \leq p^{\delta_1(L^{\Z_p}, H)} \leq  p^{\delta_1(L^{\C}, H)} \leq \left(C\log(r)+ C\right)^{\delta_1(L^{\C}, H)} \preceq \log^{\delta_1(L^{\C}, H)}(r),
	\end{equation*}
	which shows that $\RF_{L, \mathcal{P}_{H\curvearrowright L}}^{\nom} \preceq \log^{\delta_1(L^\C, H)}$ by taking the maximum over all $0<\norm{v}_L \leq r$.
	
	Now, we will argue that $\RF_{L, \mathcal{P}_{H\curvearrowright L}}^{\nom} \succeq \log^{\delta_1(L^\C, H)}$. For this, take the non-empty ideal $I_\delta$ and the bound $M_{I, \delta}\in \N$ of Lemma \ref{lem_Idelta} and $M$ of Proposition \ref{prop_Zpk_to_Zp} (where $m$ is chosen as the dimension of $L$, namely $n$). Pick a non-trivial vector $v\in I_\delta$, and define $x$ to be the product of all prime numbers smaller than or equal to $\max\{M_{I, \delta},M\}$. Define for every $l\in \N$ the element
	$$v_l = x^l (\lcm(1,2,3, \ldots, l))^{2n+1} v \in L.$$
	Note that $\norm{v_l}_L = x^l(\lcm(1,2,3, \ldots, l))^{2n+1} \norm{v}_L$ and thus by the Prime Number Theorem, there exists $C\in \N$ such that $\norm{v_l}_L \leq C^l$, see \cite[Theorem 4.3.2]{MR3559913}. We claim that $D_{L, \mathcal{P}_{H\curvearrowright L}}(v_l) \geq l^{\delta_2(L^\C, H)} = l^{\delta_1(L^\C, H)}$. If so, we will have showed that
	$$\RF_{L, \mathcal{P}_{H\curvearrowright L}}^{\nom}(C^l) \geq l^{\delta_1(L^\C, H)} = \left(\dfrac{1}{\log(C)}\right)^{\delta_1(L^\C, H)}\log^{\delta_1(L^\C , H)}(C^l),$$
	from which the general inequality $\RF_{L, \mathcal{P}_{H\curvearrowright L}}^{\nom} \succeq \log^{\delta_1(L^\C, H)}$ follows by Lemma \ref{lem_gaps_in_estimates}. 
	
	Consider the ideal that realizes $D_{L, \mathcal{P}_{H\curvearrowright L}}(v_l)$. By Lemma \ref{lem_Q_is_p_powered}, we know that we can see this ideal as an $H$-invariant ideal $I$ of $L^{\Z_{p^k}}$ for some prime power $p^k$. Surely, $v_l$ has to be non-trivial in $L^{\Z_{p^k}}$.
	
	Suppose first that $p \mid x$, i.e. $p \leq\max\{M_{I, \delta},M\}$. Then, $D_{L, \mathcal{P}_{H\curvearrowright L}}(v_l) \geq p^l \geq 2^l$, since $v_l$ can be written as $p^l\tilde{v}$ for some $\tilde{v} \in L$. For $l$ large enough, we surely have that $2^l \succeq l^{\delta_1(L^\C, H)}$. 
	
	Suppose now that $p \geq \max\{M_{I, \delta},M\}$, but $p \leq l$. We can take $l_2\in\N$ such that $p^{l_2} \leq l < p^{l_2+1}$. Now, the coefficient of $v_l$ is a multiple of $p^{l_2(2n+1)}$. Therefore,
	$$|L/I| \geq p^{l_2(2n+1)+1} \geq p^{(l_2+1)(n+1)} \geq l^{n+1} \geq l^{\delta_1(L^\C, H)}.$$
	
	Finally, suppose that $p > l$. We may assume that $v\notin pL$. Since $\gcd(p,x^l(\lcm(1,2, \ldots, l))^{2n+1}) = 1$, this implies that $v_l \notin pL$. Hence, $|L^{\Z_{p^k}}/I| = D_{L, \mathcal{P}_{H\curvearrowright L}}(v_l) \leq p^n$. Using that $v + p^kL\notin I$, Proposition \ref{prop_Zpk_to_Zp} gives us an $H$-invariant ideal $J \lhd L^{\Z_p}$ such that $v + pL \notin J$. Again, $\gcd(p,x^l(\lcm(1,2, \ldots, l))^{2n+1}) = 1$ guarantees that $v_l + pL \notin J$, and thus
	$$ D_{L, \mathcal{P}_{H\curvearrowright L}}(v_l) = |L^{\Z_{p^k}}/I| \geq |L^{\Z_p}/J|.$$ 
	Since $v+pL\in I_\delta \otimes_\Z \Z_p$, we know by Lemma \ref{lem_Idelta} that $\dim_{\Z_p} L^{\Z_p}/J \geq \delta_2(L^\C, H)$. Hence,
	$$D_{L, \mathcal{P}_{H\curvearrowright L}}(v_l) \geq p^{\delta_2(L^\C, H)} \geq l^{\delta_2(L^\C, \C)},$$
	proving the claim.
	
	By Lemma \ref{lem_link_Gui_norm}, we see that
	$$\RF_{L, \mathcal{P}_{H\curvearrowright L}}^{\nom} =  \RF_{L, \mathcal{P}_{H\curvearrowright L}}^{\Gui}$$
	by the polylogarithmic behavior of $\RF_{L, \mathcal{P}_{H\curvearrowright L}}^{\nom}$ established above.
\end{proof}
	Now, we easily derive Theorem \ref{thm_intro_exact}.
	\begin{proof}[Proof of Theorem \ref{thm_intro_exact}]
		We must argue that there exists a number $\delta\in\N$ such that $\RF_{\Gamma} = \log^\delta$, where $\Gamma$ is a finitely generated, infinite, virtually nilpotent group. This is clear. Indeed, by Corollary \ref{cor_Gtilde_to_L}, we know that $\RF_{\Gamma}$ equals $\RF^{\Gui}_{L, \mathcal{P}_{H \curvearrowright L}}$ for some Lie ring $L$ and finitely generated $H \leq \Aut(L)$. Now apply Theorem \ref{thm_RF_Pinfty}.
	\end{proof}

\subsection{Profinite invariance}

In this section the proof of Theorem \ref{thm_intro_profinite_invariant} is given. As the constant $\delta$ depends on the complex Lie algebra $L^\C$, or equivalently the complex Mal'cev completion $G^\C$, and the action of $H$ on it, it suffices to show that these are profinite invariants. As $\delta$ is defined using $H$-invariant ideals, it also suffices to show that the image of $H$ in $\Out(G^\C)$ is an invariant. We denote by $\iota: \Gamma \to \hat{\Gamma}$ the natural embedding of a residual finite group $\Gamma$ into its profinite completion $\hat{\Gamma}$.

As before, let $\Gamma$ be a virtually nilpotent group. Let $G \lhd_f \Gamma$ be an $\mathcal{I}$-group, and $H \cong \Gamma/G$.
By \cite[Proposition 3.2.2.a.]{ribes2010profinite} we can recover the profinite completion $\hat{G}$ from $\hat{\Gamma}$.
\begin{prop}
	Let $\Gamma$ be a finitely generated residually finite group. If $\pi:\hat \Gamma \rightarrow H$ is a finite discrete quotient, then the group $G=\iota^{-1}(\ker(\pi))$ has profinite completion $\hat G = \ker(\pi)$.
\end{prop}

The profinite completion $\hat G$ splits as a product $\displaystyle \prod_{p \text{ prime}} \hat G^p$ of its pro-$p$ completions $\hat G^p$, which are characteristic. In particular, if we fix a prime $p$, then the product $\displaystyle \prod_{\substack{q\neq p \\ \text{prime}}}\hat G^q$ is normal in $\hat \Gamma$. If we call its quotient $\hat \Gamma_p$, then we have the following commutative diagram with exact rows.
\[\begin{tikzcd}
	1 & G & \Gamma & H & 1 \\
	1 & {\hat G} & {\hat \Gamma} & H & 1 \\
	1 & {\hat G^p} & {\hat \Gamma_p} & H & 1
	\arrow[from=1-1, to=1-2]
	\arrow[from=1-2, to=1-3]
	\arrow[hook, from=1-2, to=2-2]
	\arrow[from=1-3, to=1-4]
	\arrow[hook, from=1-3, to=2-3]
	\arrow[from=1-4, to=1-5]
	\arrow[equals, from=1-4, to=2-4]
	\arrow[from=2-1, to=2-2]
	\arrow[from=2-2, to=2-3]
	\arrow[two heads, from=2-2, to=3-2]
	\arrow[from=2-3, to=2-4]
	\arrow[two heads, from=2-3, to=3-3]
	\arrow[from=2-4, to=2-5]
	\arrow[equals, from=2-4, to=3-4]
	\arrow[from=3-1, to=3-2]
	\arrow[from=3-2, to=3-3]
	\arrow[from=3-3, to=3-4]
	\arrow[from=3-4, to=3-5]
\end{tikzcd}\]

\noindent If we write $\textbf{Z}_p$ for the $p$-adic integers, then \cite[Chapter 9]{MR1720368} shows that, for $p$ sufficiently large,  the $\textbf{Z}_p$-completion $G^{\textbf{Z}_p}$ of $G$ (as introduced in Definition \ref{df_completion}) corresponds to the pro-$p$ completion $\hat{G}^p$. The action of $\Gamma$ on $G$ by conjugation induces on the one hand an action on the profinite completion $\hat{G}$ and thus also $\hat{G}^p$; and on the other hand an action on the $\textbf{Z}_p$-completion by extension of scalars. As both of these actions are continuous and agree on the dense subset $G$, the actions are identical.

The short exact sequences determine representations $\rho^\prime: H \to \Out(G)$ and $\rho_p: H \to \Out\left(\hat{G}^p\right)$. By taking the natural induced map $\rho: H \to \Out(\hat{G}^p)$ by the former, it is clear from the commutative diagram that the image of $\rho$ and $\rho_p$ are the same. Note that the latter is completely determined by the profinite completion $\hat{\Gamma}$ and the fixed discrete quotient $H$, which will be the key observation for the proof. 

\begin{proof}[Proof of Theorem \ref{thm_intro_profinite_invariant}]
Let $\Gamma_1, \Gamma_2$ be two finitely generated virtually nilpotent groups with isomorphic profinite completion $\hat{\Gamma}_1$ and $\hat{\Gamma}_2$. Fix $\pi_1:\Gamma_1\rightarrow H$ a finite quotient of $\Gamma_1$ such that the kernel $G_1$ is torsion-free and nilpotent as before. The map $\pi_1$ extends to a finite discrete quotient $\pi:\hat \Gamma_1\rightarrow H$ with torsion-free nilpotent kernel. As $\hat{\Gamma}_1$ and $\hat{\Gamma}_2$ are isomorphic, this map induces a surjective morphism $\pi_2:\Gamma_2\rightarrow H$, where the kernel $G_2$ is also torsion-free nilpotent.
	
	Pick a prime $p$ as before and let $\rho_i: H \to \Out(\hat{G}^p_i)$ be the actions on $G_i^{\textbf{Z}_p}=\hat G_i^p$ as before. By the observation before the proof, we obtain that $\rho_1$ and $\rho_2$ have the same image. By taking the tensor product, there are induced actions $\rho_i^\C: H\to \Out(G_i^\C)$ on $G_i^{\textbf{Z}_p}\otimes_{\textbf{Z}_p} \C = G_i^\C$. As before, $\rho_i^\C: H\to \Out(G_i^\C)$ is precisely the action induced by conjugation of $\Gamma_i$ on $G_i$, so exactly the representations occurring in Definition \ref{df_delta}. It now follows that $\mathcal P_{H\curvearrowright G_1^\C}$ and $\mathcal P_{H\curvearrowright G_2^\C}$ are identical, and thus the theorem follows from Theorems \ref{thm_RF_Pinfty} and  \ref{thm_intro_RFG_RFL}. 			
\end{proof}

\bibliographystyle{plain}
\bibliography{DMV_virt_nilpotent}
\end{document}